\documentclass[11pt]{amsart}

\usepackage{empheq}
\usepackage{amsmath,amsthm,amssymb,mathrsfs,mathtools,stmaryrd,color,relsize}

\usepackage[top=1in, bottom=1in, left=1in, right=1in]{geometry}

\usepackage{float}
\usepackage{graphicx}

\usepackage[utf8]{inputenc}
\usepackage[T1]{fontenc}

\usepackage{enumitem}
\setlist[enumerate]{itemsep=2pt,parsep=2pt,before={\parskip=2pt}}

\usepackage[colorlinks=true,hyperindex, linkcolor=red!60, pagebackref=false, citecolor=cyan, pdfpagelabels]{hyperref}
\usepackage[capitalize]{cleveref}
\crefname{equation}{}{}
\crefformat{enumi}{#2\textup{\textcolor{black}{(}#1\textcolor{black}{)}}#3}
\usepackage{appendix}
\AtBeginEnvironment{appendices}{\crefalias{section}{appendix}}

\usepackage{tikz}
\usetikzlibrary{calc}
\usetikzlibrary{math}

\usepackage{thmtools}
\usepackage{thm-restate}

\usepackage{algorithm}
\usepackage[noend]{algpseudocode}

\newtheorem{theorem}{Theorem}[section]
\newtheorem*{theorem*}{Theorem}
\newtheorem*{definition*}{Definition}
\newtheorem{proposition}[theorem]{Proposition}

\newtheorem{lemma}[theorem]{Lemma}
\newtheorem{corollary}[theorem]{Corollary}

\theoremstyle{definition}
\newtheorem{definition}[theorem]{Definition}

\newtheorem{remark}[theorem]{Remark}

\newtheorem{notation}[theorem]{Notation}

\let\originalleft\left
\let\originalright\right
\renewcommand{\left}{\mathopen{}\mathclose\bgroup\originalleft}
\renewcommand{\right}{\aftergroup\egroup\originalright}

\renewcommand{\Re}{\operatorname{Re}}
\renewcommand{\Im}{\operatorname{Im}}
\newcommand{\R}{\mathbb{R}}
\newcommand{\C}{\mathbb{C}}
\newcommand{\Z}{\mathbb{Z}}
\newcommand{\N}{\mathbb{N}}

\renewcommand{\geq}{\geqslant}
\renewcommand{\leq}{\leqslant}
\renewcommand{\phi}{\varphi}
\newcommand{\eps}{\varepsilon}
\newcommand{\vol}{\operatorname{vol}}

\renewcommand{\Pr}[1]{\mathbb{P}\left(#1\right)}
\newcommand{\Prbis}{\mathbb{P}}
\newcommand{\E}[1]{\mathbb{E}\left[#1\right]}
\DeclareMathOperator*{\Ebis}{{}\mathbb{E}}%can add \mathlarger

\newcommand{\norm}[1]{\left\|#1\right\|}
\newcommand{\inner}[1]{\left\langle#1\right\rangle}
\newcommand{\abs}[1]{\left|#1\right|}
\newcommand{\ind}[1]{\mathbf{1}_{#1}}
\newcommand{\spmod}[1]{\,(\mathrm{mod}\, #1)}
\newcommand{\order}{\mathrm{ord}}
\newcommand{\Nb}{\mathcal{N}}

\newcommand{\theM}{{M_*}}

\newcommand{\n}{\mathbf{n}}
\renewcommand{\b}{\mathbf{b}}
\newcommand{\x}{\mathbf{x}}

\newcommand{\z}{\mathbf{z}}

\renewcommand{\L}{\mathcal{L}}
\newcommand{\M}{\mathcal{M}_{\scriptscriptstyle{\mathrm{RSA}}}}

\newcommand{\mgp}[1]{(\Z/#1\Z)^{\times}}

\newcommand{\rect}[2]{\left\{s\in \C\ : \ {#1}< \Re{s} \leq 1,\ \abs{\Im{s}} \leq {#2}\right\}}

\begin{document}

\title[Unconditional correctness of recent quantum algorithms]{Unconditional correctness of recent quantum algorithms for factoring and computing discrete logarithms}

\author{C\'edric Pilatte}
\address{Mathematical Institute, University of Oxford.}
\email{cedric.pilatte@maths.ox.ac.uk}

\begin{abstract}
	In 1994, Shor introduced his famous quantum algorithm to factor integers and compute discrete logarithms in polynomial time. In 2023, Regev proposed a multi-dimensional version of Shor's algorithm that requires far fewer quantum gates. His algorithm relies on a number-theoretic conjecture on the elements in $(\mathbb{Z}/N\mathbb{Z})^{\times}$ that can be written as short products of very small prime numbers. We prove a version of this conjecture using tools from analytic number theory such as zero-density estimates. As a result, we obtain an unconditional proof of correctness of this improved quantum algorithm and of subsequent variants.
\end{abstract}

\maketitle

\section{Introduction}

\subsection{Context and quantum computing results}
Public key cryptography has become a crucial element of our global digital communication infrastructure. Notable examples include the Diffie-Hellman key exchange~\cite{DH} and the RSA (Rivest-Shamir-Adleman) cryptosystem~\cite{RSA}, which rely on the difficulty of finding discrete logarithms and factoring large numbers, respectively.

However, in 1994, Peter Shor~\cite{shor} developed an algorithm capable of efficiently solving these problems using a quantum computer.

\begin{theorem*}[Shor\protect\footnote{The version of Shor's algorithm stated here incorporates some small improvements from \cite{shorbis} (bounded number of calls) and \cite{annalsmult} (fast integer multiplication). The number of qubits in Shor's algorithm can be brought down to $O(n)$, though this typically requires a larger number of quantum gates (see e.g.~\cite{gid,fastmult}).}]
	There is a quantum circuit having $O(n^{2} \log n)$ quantum gates and $O(n \log n)$ qubits with the following property. There is a classical randomised polynomial-time algorithm that solves the factoring problem
	\begin{align*}
		 & \mathbf{Input: }\text{ a composite integer $N \leq 2^n$} \\
		 & \mathbf{Output: }\text{ a non-trivial divisor of $N$}
	\end{align*}
	using $O(1)$ calls to this quantum circuit, and succeeds with probability $\Theta(1)$.
\end{theorem*}

Shor's original article \cite{shor} also includes a similar algorithm to solve the discrete logarithm problem. This advancement prompted the development of post-quantum cryptography, such as lattice-based cryptography, to ensure the security of our communications in the face of potential quantum computing breakthroughs~\cite{postquantum}.

Recently, Regev~\cite{regev} devised a multidimensional variant of Shor's factoring algorithm that reduces the size of the quantum circuit, i.e.~the number of quantum gates, to $O(n^{3/2}\log n)$. A distinctive feature of Regev's algorithm is that the quantum circuit must be called $O(\sqrt{n})$ times, rather than a constant number of times for Shor's algorithm. This is not considered to be a serious drawback as the various complexity parameters of the quantum circuit are far more relevant metrics in quantum computing.

However, this remarkable work of Regev initially came with two main limitations.

First, the original version of Regev's algorithm requires $O(n^{3/2})$ qubits -- many more than Shor's algorithm. The reason for this ultimately lies in the difficulty of performing classical computations on quantum computers in a space-efficient manner, due to the need for any quantum computation to be \emph{reversible}. In a subsequent paper, Ragavan and Vaikuntanathan~\cite{MIT} improved Regev's algorithm to use only $O(n\log n)$ qubits, while maintaining the circuit size at $O(n^{3/2}\log n)$ quantum gates. This work, inspired by ideas of Kaliski~\cite{kaliski}, essentially matches the space cost of Shor's algorithm.

The second issue, which we address in this paper, is that Regev's algorithm~\cite{regev} does not have theoretical guarantees, unlike Shor's algorithm. The correctness of Regev's algorithm is based on an \emph{ad hoc} number-theoretic conjecture which we describe below. This unproven assumption cannot be avoided as it lies at the core of Regev's improvement on the circuit size. The variant of Ragavan and Vaikuntanathan~\cite{MIT} also crucially relies on this conjecture.

In this paper, we prove a version of Regev's conjecture. This allows us to unconditionally prove the correctness of (slightly modified versions of) the algorithms of Regev~\cite{regev} and Ragavan and Vaikuntanathan~\cite{MIT}. More precisely, we obtain the following algorithmic result.

\begin{restatable}{theorem}{thmfactoring}
	\label{thm:factoring}
	There is a quantum circuit having $O(n^{3/2} \log^3 n)$ quantum gates and $O(n \log^3 n)$ qubits with the following property. There is a classical randomised polynomial-time algorithm that solves the factoring problem
	\begin{align*}
		 & \mathbf{Input: }\text{ a composite integer $N \leq 2^n$} \\
		 & \mathbf{Output: }\text{ a non-trivial divisor of $N$}
	\end{align*}
	using $O(\sqrt{n})$ calls to this quantum circuit, and succeeds with probability $\Theta(1)$.
\end{restatable}

This unconditionally establishes the results in \cite{MIT,regev} up to logarithmic factors.

Regev's algorithm was adapted by Eker{\aa} and G{\"a}rtner~\cite{discretelogs} to the discrete logarithm problem. Their paper~\cite{discretelogs} also uses the space-saving arithmetic of Ragavan and Vaikuntanathan~\cite{MIT} to obtain a quantum circuit with $O(n^{3/2}\log n)$ gates and $O(n\log n)$ qubits for computing discrete logarithms.

Once more, the correctness of the algorithm by Eker{\aa} and G{\"a}rtner~\cite{discretelogs} relies on an unproven hypothesis, which can be viewed as a stronger form of Regev's conjecture. Our methods also apply to this stronger statement (again, with minor technical adjustments). Thus, we get an analogue of \cref{thm:factoring} for the discrete logarithm problem, i.e.~an unconditional proof of the results in \cite{discretelogs} up to logarithmic factors.

\begin{restatable}{theorem}{thmlogarithm}
	\label{thm:logarithm}
	There is a quantum circuit having $O(n^{3/2} \log^3 n)$ quantum gates and $O(n \log^3 n)$ qubits with the following property. There is a classical randomised polynomial-time algorithm that solves the discrete logarithm problem
	\begin{align*}
		 & \mathbf{Input: }\text{ an integer $N \leq 2^n$ and elements $g,y \in \mgp{N}$ such that $y\in \inner{g}$} \\
		 & \mathbf{Output: }\text{ an integer $x$ such that $g^x \equiv y \pmod{N}$}
	\end{align*}
	using $O(\sqrt{n})$ calls to this quantum circuit, and succeeds with probability $\Theta(1)$.
\end{restatable}

Moreover, the slight technical modifications that we need to introduce to make all these quantum algorithms unconditional are compatible with the error-correction results of \cite{discretelogs,MIT}. Finally, our results extend to further variants of Regev's algorithm, such as computing discrete logarithms or multiplicative orders modulo $N$ for several elements simultaneously (see~\cite{discretelogs}). %actually order finding is equivalent to factoring, see \cite[Appendix~A4.3]{CNbook}

\begin{remark}
	\label{rem:saveonemorelog}
	By being slightly more careful in the space usage of our quantum algorithms, it is possible to reduce the number of qubits to $O(n\log^2 n)$ for \cref{thm:factoring,thm:logarithm}. We will not show this here in order to keep the quantum computing prerequisites to a minimum.
\end{remark}

\subsection{An overview of Regev's algorithm}

Let us start by recalling the basic idea behind Shor's factoring algorithm~\cite{shor}.

Let $N$ be the integer to be factored, which can be assumed to be odd and with at least two distinct prime factors. Shor proved that there is an efficient quantum algorithm to find the multiplicative order of any invertible element modulo $N$. With this in hand, factoring $N$ is straightforward. The first step is to pick an integer $1 < a < N$ uniformly at random. If $\gcd(a,N)>1$, the factoring task is trivial. Otherwise, with probability at least~$1/2$, the order of $a$ modulo $N$, denoted $r$, is even and satisfies $a^{r/2}\not\equiv\pm1$ (i.e.~$a^{r/2}$ is a non-trivial square root of $1$ modulo $N$). Whenever this is the case, we get a non-trivial divisor of~$N$, namely $\gcd(N, a^{r/2}-1)$.

The reason that Shor's algorithm uses $O(n^2\log n)$ quantum gates, where $n := \lceil \log_2 N\rceil$, comes from the part of the quantum circuit that performs modular exponentiation. Consider the task of computing a power $a^M\spmod{N}$ on a classical computer, where $1<a<N-1$ and $M \leq N$. This computation can be performed efficiently using the well-known square-and-multiply method, which involves $O(\log M)$ multiplications of two integers modulo $N$. Since two $n$-bit integers can be multiplied in time $O(n\log n)$ by~\cite{annalsmult}, the complexity of this modular exponentiation problem is thus $O(n^2\log n)$.\footnote{Note that, even if $a$ is a small integer, say $a=2$, this procedure still takes time $O(n^2\log n)$, because most multiplications will involve $n$-bit integers.} For Shor's algorithm, a similar modular exponentiation needs to be performed quantumly, which requires $O(n^2\log n)$ quantum gates.

Regev's improvement of Shor's algorithm is made possible by combining two key ideas.

The first idea in Regev's algorithm is to work in \emph{higher-dimensional} space and replace the random parameter $a$ by several integers $b_1, \ldots, b_d$ (chosen in a specific way, as we explain below). Eventually, the optimal choice of dimension turns out to be ${d \asymp \sqrt{\log N}}$. Similarly to Shor's algorithm, the problem of factoring $N$ easily reduces to the task of finding a vector $(e_1, \ldots, e_d) \in \Z^d$ such that $\prod_{i=1}^d b_i^{e_i}$ is a non-trivial square root of $1$ modulo $N$.

Using additional tools such as the LLL lattice reduction algorithm, Regev~\cite{regev} generalised Shor's quantum algorithm to efficiently find all such vectors $(e_1, \ldots, e_d)$ in a ball of radius $N^{O(1/d)}$ centred at the origin.

Regev's algorithm only succeeds if there indeed exists a vector $v = (e_1, \ldots, e_d) \in \Z^d$ such that
\begin{enumerate}[label=\textcolor{black}{(}\roman*\textcolor{black}{)}]
	\item \label{item:cond1} $\norm{v}_2 \leq N^{O(1/d)}$, and
	\item \label{item:cond2} $\prod_{i=1}^d b_i^{e_i}$ is a non-trivial square root of $1$ modulo $N$.
\end{enumerate}

If the parameters $b_1, \ldots, b_d$ are ``sufficiently multiplicatively independent'' modulo $N$, one might heuristically expect the existence of a vector $v = (e_1, \ldots, e_d)$ satisfying \ref{item:cond1} and \ref{item:cond2}. This is, roughly speaking, what Regev needs to assume. To state his conjecture properly, it remains to specify how the parameters $b_1, \ldots, b_d$ are chosen. This is a crucial point, as the idea of working in dimension $d$ does not, on its own, offer any advantage over Shor's algorithm.

The second key insight of Regev is to choose $b_1, \ldots, b_d$ to be \emph{very small} compared to $N$. We mentioned that the most costly part of Shor's quantum circuit lies in the modular exponentiation step. Similarly, the number of gates of Regev's circuit is dominated by the cost of computing an expression of the form
\begin{equation}
	\label{eq:monomialexpr}
	\prod_{i=1}^d b_i^{M_i}\spmod{N}
\end{equation}
in the quantum setting, where the exponents $M_1, \ldots, M_d$ are $\leq N^{O(1/d)}$. Regev observed that~\cref{eq:monomialexpr} can be computed classically in time $O(n^{3/2} \log n)$ if $|b_i| \leq (\log N)^{O(1)}$ for all $i$ (for $d \asymp \sqrt{n} \asymp \sqrt{\log N}$). This can be achieved by cleverly ordering the intermediate multiplications so that most of them involve integers with much fewer than $n$ bits. This classical procedure can then be turned into a quantum version that uses $O(n^{3/2}\log n)$ gates.

Shor's algorithm corresponds to the $d=1$ case of Regev's algorithm where the parameter $b_1$ is an element of $\mgp{N}$ chosen uniformly at random. In this case, simple considerations from elementary number theory immediately imply the existence of an integer $e_1$ satisfying \ref{item:cond1}~and~\ref{item:cond2} with probability $\gg 1$. This would hold more generally for $d\geq 1$ if $b_1, \ldots, b_d$ were independent, uniformly distributed random elements of $\mgp{N}$. However, for Regev's algorithm, the $b_i$'s are far from uniformly distributed in $\mgp{N}$ as they are constrained to lie in a very small subset of $\mgp{N}$.

\subsection{The number-theoretic conjecture behind Regev's algorithm} Regev's algorithm \cite{regev} and its space-efficient variant \cite{MIT} crucially rely on the possibility of choosing very small integers $b_1, \ldots, b_d$ such that $\prod_{i=1}^d b_i^{e_i}$ is a non-trivial square root of $1$ modulo $N$ for some $e_1, \ldots, e_d$ with $|e_i| \leq N^{O(1/d)}$. We remind the reader that $n \asymp \log N$ and $d \asymp \sqrt{n}$.

The least restrictive bound\footnote{The bound stated in Regev's paper~\cite{regev} is $|b_i| \leq (\log N)^{O(1)}$, but this condition can be relaxed somewhat, as observed by Ragavan (private communication).} on the $b_i$'s that still allows for a quantum circuit with $\tilde{O}(n^{3/2})$ gates is to have $|b_i| \leq \exp(\tilde{O}(d))$, where the $\tilde{O}(\cdot)$ notation possibly hides a factor $(\log n)^{O(1)}$. Moreover, it is natural to set the $b_i$'s to be prime numbers (as in \cite{discretelogs,MIT,regev}) in order to avoid obvious multiplicative relations between them.

In this paper, we choose $b_1, \ldots, b_d$ to be independent random prime numbers less than $d^{10^3d}$. With these parameters, we can now state a version of Regev's conjecture\footnote{Compare with \cite[Theorem~1.1]{regev} or \cite[Conjecture~1]{MIT}.} that follows from our results.%previously \cite[Conjecture~3.1]{MIT}

\begin{corollary}
	\label{cor:efficientrep}
	Let $N>2$ be an integer. Let $d := \lceil\!\sqrt{\log N} \rceil$ and $X := d^{10^3d}$.

	Let $\b_1, \ldots, \b_d$ be i.i.d.~random variables, each uniformly distributed in the set of primes $\leq X$ not dividing $N$.

	Then, with high probability, every $x \in \inner{\b_1, \ldots, \b_d}$ can be expressed as
	\begin{equation*}
		x \equiv \prod_{i=1}^d \b_i^{e_i} \pmod{N}
	\end{equation*}
	for some integers $e_i$ with $|e_i|\leq e^{O(d)}$ for all $1\leq i\leq d$.
\end{corollary}

The analogue of Regev's algorithm for the discrete logarithm problem, proposed by Eker{\aa} and G{\"a}rtner~\cite{discretelogs}, relies on a stronger version of Regev's conjecture. It concerns the geometry of a certain lattice $\L$ which encodes all multiplicative dependencies between the $\b_i$'s (modulo $N$). We prove the following version of it.\footnote{Compare with \cite[Assumption~1]{discretelogs}, or Conjecture~E.1 in the full version of \cite{MIT}.}%previously \cite[Conjecture~E.1]{MIT}

\begin{corollary}
	\label{cor:shortbasis}
	Let $N>2$ be an integer. Let $d := \lceil\!\sqrt{\log N} \rceil$ and $X := d^{10^3d}$.

	Let $\b_1, \ldots, \b_d$ be i.i.d.~random variables, each uniformly distributed in the set of primes $\leq X$ not dividing $N$.

	Let $\L$ be the random lattice defined by
	\begin{equation}
		\label{eq:latticeintro}
		\L := \Big\{ (e_1, \ldots, e_d) \in \Z^d \ : \ \prod_{i=1}^d \b_i^{e_i} \equiv 1 \!\!\pmod{N} \Big\}.
	\end{equation}
	Then, with high probability, this lattice $\L$ has a basis consisting of vectors of Euclidean norm $\leq e^{O(d)}$.
\end{corollary}

The main technical result of this paper is \cref{thm:shortbasis}, of which \cref{cor:shortbasis} is a special case. In~turn, \cref{cor:efficientrep} is a direct consequence of \cref{cor:shortbasis}.

\subsection{Subgroup obstructions}
Let $N$, $d$ and $\b_1, \ldots, \b_d$ be as in \cref{cor:efficientrep}. This corollary shows that, with high probability, every element $x$ in the subgroup generated by $\b_1, \ldots, \b_d$ can be written as a \emph{short} product of these $\b_i$ modulo $N$. To obtain the full strength of Regev's assumption \cite[Conjecture~1]{MIT}, it would be necessary to show that this subgroup $\inner{\b_1, \ldots, \b_d}$ contains a non-trivial square root of $1$ modulo $N$.%previously \cite[Conjecture~3.1]{MIT}

Unfortunately, it is not possible to prove this in full generality without considerable advances on the well-known \emph{least quadratic non-residue problem} in number theory. For any prime number $p$, write $n(p)$ for the smallest positive integer $a$ which is a quadratic non-residue (i.e.~not a square) modulo~$p$. The best known asymptotic upper bound for $n(p)$ is Burgess's classical result \cite{burgess} that
\begin{equation*}
	n(p) \ll_{\eps} p^{\frac{1}{4\sqrt{e}}+\eps}.
\end{equation*}
For simplicity, suppose that the integer $N$ to be factored is a product of two equally-sized primes ${p_1, p_2\equiv 3 \pmod{4}}$. Recall that, for Regev's algorithm, the parameters $\b_i$ are chosen to be less than $\exp(\tilde{O}(d))$. With current techniques, we cannot rule out the possibility that both $n(p_1)$ and $n(p_2)$ are larger than this threshold. If this is the case, all $\b_i$ will be quadratic residues modulo both $p_1$ and $p_2$, which implies that $\inner{\b_1, \ldots, \b_d}$ is contained in the subgroup of squares modulo $N$. In particular, since $p_1, p_2\equiv 3 \pmod{4}$, the subgroup $\inner{\b_1, \ldots, \b_d}$ does not contain any non-trivial square roots of $1$ modulo $N$.

This issue can be approached from several perspectives.

\begin{enumerate}
	\item Assuming the Generalised Riemann Hypothesis, Ankeny proved the much stronger bound $n(p) \ll (\log p)^2$ \cite{ankeny}. This suggests that quadratic residues may no longer be an obstruction in this case, and indeed, under GRH, it is straightforward to show that $\inner{\b_1, \ldots, \b_d}$ contains a non-trivial square root of $1$ modulo $N$ with high probability. Assuming GRH would also considerably simplify the proof of \cref{thm:shortbasis} and remove a logarithmic factor for the gate and qubit costs in \cref{thm:factoring,thm:logarithm}. However, in this paper we seek fully unconditional results.
	\item It is possible to prove the following partial result unconditionally: for \emph{almost all} odd $N$ with at least two distinct prime factors, the subgroup $\inner{\b_1, \ldots, \b_d}$ contains a non-trivial square root of $1$ modulo $N$ with high probability (with $d = d(N)$ and $\b_i = \b_i(N)$ as in \cref{cor:efficientrep}). By \cref{cor:efficientrep}, this implies that Regev's original algorithm (with these parameters $\b_i$) finds a non-trivial divisor of $N$ with high probability, for \emph{almost all} $N$. More precisely, the set $E\subset \N$ of exceptional values of $N$ for which the above does not hold can be shown to satisfy
	      \begin{equation*}
		      \abs{E \cap [1, x]} \leq \exp\!\big((\log x)^{1/2+o(1)}\big)
	      \end{equation*}
	      for $x\geq 1$. This partial result can be proved using character sums, the Landau-Page theorem \cite[Corollary~11.10]{MV} and a grand zero-density estimate \cite[Theorem~1~(1.8)]{jutila}, by means of a case analysis depending on the prime factorisation of $N$.
	\item In the present work, we follow a different approach to obtain a completely unconditional result that applies to \emph{all} $N$, by slightly modifying the algorithm itself. While the key to the efficiency of Regev's algorithm is that the $\b_i$'s are small, one can tolerate a bounded number of large $\b_i$'s.\footnote{In fact, this observation was already needed to adapt Regev's algorithm to the discrete logarithm problem~\cite{discretelogs}.} We use this extra flexibility to overcome the subgroup obstructions. The simplest way to proceed is to allow for one of the parameters, say $\b_1$, to be uniformly distributed in $\mgp{N}$. As with Shor's algorithm, this ensures that the subgroup $\inner{\b_1}$, and hence also $\inner{\b_1, \ldots, \b_d}$, contains a non-trivial square root of $1$ with probability $\gg 1$.
\end{enumerate}

\begin{remark}
	\label{rem:deterministic}
	In his paper~\cite{regev}, Regev proposed to select the parameters $b_1, \ldots, b_d$ deterministically, for example by setting $b_i$ to be the $i$-th smallest prime for $1\leq i\leq d$. While such a deterministic choice of parameters is likely to work for \emph{almost all} $N$, this seems to be very difficult to prove.
\end{remark}

\begin{remark}
	\label{rem:integersnotprimes}
	In this paper, the parameters $\b_i$ are chosen to be random primes (not dividing~$N$) under a certain threshold $X$. Another natural choice would have been to let $\b_1, \ldots, \b_d$ be i.i.d.~random variables uniformly distributed in the set of all integers $\leq X$ coprime to $N$. The proofs in this paper would carry over to this framework if an analogue of \cref{lem:charbound} for short character sums $\sum_{n\leq x} \chi(n)$ were available. Assuming the Generalised Riemann Hypothesis, the work of Granville and Soundararajan~\cite[Theorem~2]{GS} gives a bound of the required strength. However, this approach does not seem sufficient to obtain an unconditional result (the bounds in \cite{GS} become too weak when $L(s, \chi)$ has zeroes close to the line $\Re s = 1$).
\end{remark}

\subsection{Structure of the paper}
We begin in \cref{sec:toyversion} with a proof of a simple special case of our main theorem. This is intended to illustrate our overall strategy and motivate the need for the more technical approach required in the general case. The purpose of this section is to build intuition; it is not formally required for the main argument.

Our main technical result is \cref{thm:shortbasis}, which shows that lattices $\L$ similar to that in \cref{cor:shortbasis} have a basis of short vectors with high probability. Using simple geometry of numbers (see \cref{sec:geometrynumbers}), we reduce this problem to estimating the number of lattice points in balls of growing radii. Unfortunately, we are unable to obtain a suitable lattice point count for $\L$ directly. We resolve this by considering a different lattice $\L_M$ from the start of the argument (using the lemmas in \cref{sec:restrictionM}). In \cref{sec:charactercount}, we expand the lattice point count for $\L_M$ in terms of Dirichlet characters modulo $N$. This produces a main term, which can be estimated precisely, and an error term. The heart of the proof lies in using a zero-density estimate for Dirichlet characters modulo $N$ to bound this error term unconditionally. Finally, we prove our quantum algorithmic applications (\cref{thm:factoring,thm:logarithm}) in \cref{sec:proofsquantum}.

\subsection{Notation}
We write $f\ll g$ or $f = O(g)$ if $|f| \leq Cg$ for some absolute constant $C>0$. If instead, $C$ depends on a parameter $\theta$, we write $f \ll_{\theta} g$ or $f = O_{\theta}(g)$. The notation $f\asymp g$ or $f = \Theta(g)$ means that $f\ll g$ and $g\ll f$.

A \emph{character} of a finite abelian group $G$ is a homomorphism $\chi : G\to \C^{\times}$. The \emph{order} of $\chi$, denoted by $\order(\chi)$, is the least positive integer $n$ such that $\chi^n$ is the trivial character $1$. The group of all characters of $G$ is denoted by~$\widehat{G}$. If $g_1, \dots, g_k\in G$, we write $\inner{g_1, \ldots, g_k}$ for the subgroup of $G$ generated by these elements.

A non-trivial square root of $1$ modulo $N$ is an element $a\in \mgp{N}$ such that $a^2 \equiv 1\pmod N$ and $a\not\equiv \pm1 \pmod N$. Euler's totient function and the prime counting function are denoted by $\varphi$ and $\pi$, respectively. We use the notation $\norm{\cdot }_2$ for the Euclidean norm. We write $\log$ for the natural logarithm and $\log_2$ for the logarithm in base $2$.

Random variables are typically written in bold font, such as $\b_1, \ldots, \b_d$. A statement will be said to occur \emph{with high probability} if it holds with probability tending to $1$ as $N\to \infty$.

We refer the reader to \cite{CNbook} for a detailed textbook on quantum computing, and \cite[Chapter~2]{postquantum} for a brief overview.

\section*{Acknowledgements}

The author is supported by the Oxford Mathematical Institute and a Saven European Scholarship. I would like to thank Jane Street for awarding me one of their Graduate Research Fellowships. I am indebted to my advisors, Ben Green and James Maynard, for their invaluable advice and support. Finally, I wish to thank Seyoon Ragavan for his kind explanations, and in particular for showing me that the upper bound for the parameters in Regev's algorithm could be relaxed considerably.

\section{An illustrative special case}
\label{sec:toyversion}

In this section, we prove a weak version of our main theorem in a simplified setting. This special case can be used to obtain an unconditional factoring algorithm for integers $N$ of a special form: RSA moduli that are products of two ``safe primes''.\footnote{Note that this restrictive assumption is generally not satisfied by RSA moduli used in practice.}

\begin{definition}[RSA moduli with safe primes]
	\label{def:RSAmod}
	Let $\M$ be the set of integers $N$ of the form $N = P\cdot Q$, where $P,Q\geq N^{1/4}$ are distinct primes such that $(P-1)/2$ and $(Q-1)/2$ are also prime.
\end{definition}

For $N\in \M$, the multiplicative group $\mgp{N}$ has a very convenient structure, allowing us to give a relatively short proof of the following proposition.

\begin{proposition}
	\label{prop:toyversion}
	Let $N\in \M$ be a sufficiently large integer. Let $d := \lceil\!\sqrt{\log N} \rceil$ and $X := d^{d}$.

	Let $\b_1, \ldots, \b_d$ be i.i.d.~random variables, each uniformly distributed in the set of primes less than~$X$. Let $\b_0$ be a random variable uniformly distributed in $\mgp{N}$, independent of the $\b_i$'s.

	Let $x\in \mgp{N}$. Then, with probability $\gg 1$, there are integers $0\leq e_0, \ldots, e_d\leq e^{O(d)}$ such that
	\begin{equation*}
		\prod_{i=0}^d \b_i^{e_i} \equiv x \pmod{N}.
	\end{equation*}
\end{proposition}

Applying \cref{prop:toyversion} (with $x$ being a non-trivial square root of $1$ modulo $N$) shows that Regev's factoring algorithm, with random parameters $\b_0, \ldots, \b_d$ as in that proposition\footnote{Again, the presence of one potentially large parameter $\b_0$ circumvents the subgroup obstructions discussed in the introduction, while not significantly affecting the efficiency of the quantum algorithm.}, admits an unconditional proof of correctness for integers $N\in \M$.

It should be noted that \cref{thm:shortbasis} not only removes the restriction $N\in \M$, it also provides a stronger geometric conclusion required for applications to the discrete logarithm problem.

\begin{notation}
	\label{not:toy}
	For the rest of this section, we fix $N$, $d$, $X$, $\b_0, \ldots, \b_d$ and $x$ as in the statement of \cref{prop:toyversion}. In particular, $N\in \M$ is assumed to be sufficiently large. Moreover, there are primes $P$, $Q$, $P'$ and $Q'$ such that $N = P\cdot Q$, $P = 2P'+1$ and $Q = 2Q'+1$. Hence,
	\begin{equation}
		\label{eq:structureGtoy}
		G := \mgp{N} \cong (\Z/2\Z)^2 \times \Z/P'\Z \times \Z/Q'\Z \cong \widehat{G}.
	\end{equation}
\end{notation}

To prove \cref{prop:toyversion}, we require the following lemma from analytic number theory, that ultimately relies on a zero-density estimate for Dirichlet $L$-functions.
\begin{lemma}
	\label{lem:ANTtoy}
	For every $j\geq 0$, the number of characters $\psi \in \widehat{G}$ such that
	\begin{equation*}
		\Big\lvert \! \Ebis_{p\leq X} \psi(p) \Big\rvert \geq e^{-j}
	\end{equation*}
	is at most $e^{O(d)}N^{O(j/\log X)}$. Here, $\Ebis_{p\leq X}$ denotes the average over primes $p\leq X$.
\end{lemma}
We omit the proof of \cref{lem:ANTtoy}, as it is somewhat tangential to the aim of this explanatory section.\footnote{This lemma is a consequence of \cref{lem:charbound,prop:ZDEmodN}; see \cref{sec:largeorder} and \cref{sec:appendix} for further details.}

\begin{proof}[Proof of \cref{prop:toyversion}]
	Let $H = \lfloor e^{Cd} \rfloor$, where $C$ is a large absolute constant to be chosen later.

	We will show that, with probability $\gg 1$, there exist integers $0\leq h_1, \ldots, h_d < H$ such that
	\begin{equation*}
		x \equiv \b_0 \b_1^{h_1}\cdots \b_d^{h_d} \pmod{N},
	\end{equation*}
	which clearly implies the desired conclusion. By orthogonality of characters modulo $N$, we have
	\begin{equation}
		\label{eq:orthogonalitytoy}
		\sum_{0\leq h_1, \ldots, h_d < H} \ind{\b_0 \b_1^{h_1}\cdots \b_d^{h_d} \equiv x \spmod{N}} = \frac{1}{\varphi(N)} \sum_{\chi \in \widehat{G}} \chi(\b_0x^{-1}) \prod_{i=1}^d \sum_{0\leq h < H} \chi^{h}(\b_i).
	\end{equation}

	Let $\chi_0$ be the trivial character, and $\chi_1, \chi_2, \chi_3$ be the three characters of order $2$ in $\widehat{G}$. Each of these four characters $\chi_j$ satisfies $\chi_j(\b_i)\in \{\pm 1\}$ for all $i$, so the term $\sum_{0\leq h < H} \chi_j^{h}(\b_i)$ is either equal to $H$ or lies in $\{0,1\}$. Hence, the contribution of these four characters to \cref{eq:orthogonalitytoy} is
	\begin{equation}
		\label{eq:maintermtoy}
		\frac{H^d}{\varphi(N)} + \frac{1}{\varphi(N)} \sum_{j=1}^3 \chi_j(\b_0x^{-1})\, C_j(\b_1,\ldots, \b_d)
	\end{equation}
	for some quantities $C_j(\b_1,\ldots, \b_d)\geq 0$. Since $\b_0$ is uniformly distributed in $G$, by \cref{eq:structureGtoy} we have
	\begin{equation*}
		\Pr{\chi_1(\b_0x^{-1}) = \chi_2(\b_0x^{-1}) = \chi_3(\b_0x^{-1}) = 1} = \frac{1}{4},
	\end{equation*}
	as this is the probability that $\b_0x^{-1}$ lies in the subgroup of squares modulo $N$, which has index~$4$ in~$G$. Therefore, with probability $\geq 1/4$, the expression \cref{eq:maintermtoy} is bounded below by $H^d/\varphi(N)$.

	As a result, \cref{prop:toyversion} reduces to showing that the contribution of the remaining characters $\chi\in \widehat{G}\setminus \{\chi_0,\chi_1,\chi_2,\chi_3\}$ to \cref{eq:orthogonalitytoy} is, with high probability, at most $H^d/(2\varphi(N))$ in absolute value. By Chebyshev's inequality, it suffices to prove that
	\begin{equation}
		\label{eq:secondmomenttoy}
		\Ebis\!\big[\abs{\z}^2\big] = o\big(H^{2d}\big),
	\end{equation}
	where $\z$ is the random variable defined by
	\begin{equation*}
		\z := \sum_{\substack{\chi\in \widehat{G} \\ \order(\chi) > 2}} \Bigg\lvert \prod_{i=1}^d \sum_{0\leq h < H} \chi^{h}(\b_i)\Bigg\rvert.
	\end{equation*}
	By the triangle inequality for the $L^2$-norm and the definition of the random variables $\b_i$, we have
	\begin{equation*}
		\Ebis\!\big[\abs{\z}^2\big]^{1/2} \leq \sum_{\substack{\chi\in \widehat{G} \\ \order(\chi) > 2}} \Bigg( \Ebis\Bigg[ \bigg\lvert \prod_{i=1}^d \sum_{0\leq h < H} \chi^{h}(\b_i)\bigg\rvert^2\Bigg]\Bigg)^{1/2} = \sum_{\substack{\chi\in \widehat{G} \\ \order(\chi) > 2}} \Bigg(  \Ebis_{p\leq X} \bigg\lvert \! \sum_{0\leq h < H} \chi^h(p) \bigg\rvert^2 \Bigg)^{d/2}.
	\end{equation*}
	Expanding the square and exchanging the order of summation, we obtain
	\begin{equation*}
		\Ebis_{p\leq X} \bigg\lvert\!\sum_{0\leq h < H} \chi^h(p) \bigg\rvert^2 \leq \sum_{0\leq h_1,h_2 < H}\, \abs{\Ebis_{p\leq X} \chi^{h_1 - h_2}(p)} \leq H \sum_{|h| < H} \, \abs{\Ebis_{p\leq X} \chi^{h}(p)}.
	\end{equation*}
	Therefore,
	\begin{equation}
		\label{eq:intermediatetoy}
		\Ebis\!\big[\abs{\z}^2\big]^{1/2} \leq H^{d/2} \sum_{\substack{\chi\in \widehat{G} \\ \order(\chi) > 2}} \Bigg( \sum_{|h| < H} \abs{\Ebis_{p\leq X} \chi^{h}(p)} \Bigg)^{d/2}.
	\end{equation}

	To estimate the right-hand side of \cref{eq:intermediatetoy}, we partition the set of characters as follows. For $j\geq 0$, let
	\begin{equation*}
		E_j := \Big\{\chi \in \widehat{G}\setminus \{\chi_0,\chi_1,\chi_2,\chi_3\} \ :\  \max_{0< |h| < H} \Big\lvert\! \Ebis_{p\leq X} \chi^{h}(p) \Big\rvert \in \big(e^{-j-1}, e^{-j}\big] \Big\}.
	\end{equation*}
	Notice that, by \cref{eq:structureGtoy} and \cref{def:RSAmod}, every character $\chi$ of order greater than $2$ has order at least $\min(P', Q') \gg N^{1/4}$. Hence, provided $N$ is sufficiently large, for any $\chi$ of order greater than~$2$, the characters $\chi^h$ for $0<|h|<H$ are all distinct. In~addition, for any fixed $\psi \in \widehat{G}\setminus \{\chi_0,\chi_1,\chi_2,\chi_3\}$ and $0<|h|<H$, there are at most four characters $\chi$ such that $\chi^h = \psi$. By \cref{lem:ANTtoy}, these observations imply that
	\begin{equation}
		\label{eq:sizeEjtoy}
		\abs{E_j} \ll H e^{O(d)} N^{O(j/\log X)} \ll H e^{O(d)} e^{O(jd/\log d)}
	\end{equation}
	for all $j\geq 0$.\footnote{It is here that we crucially relied on the particular structure of $\mgp{N}$ for $N\in \M$. These arguments can be extended somewhat, but additional ideas are required to handle the general case (the most difficult case being when $N$ has many prime factors).}

	Let $j\geq 0$ and let $\chi \in E_j$. By \cref{lem:ANTtoy}, there are at most $e^{O(d)}$ integers $0\leq |h|<H$ such that $\abs{\Ebis_{p\leq X} \chi^{h}(p)} > d^{-2}$. By definition of $E_j$, all other integers $0<|h|<H$ satisfy
	\begin{equation*}
		\abs{\Ebis_{p\leq X} \chi^{h}(p)} \leq \min(d^{-2}, e^{-j}).
	\end{equation*}
	Hence,
	\begin{equation*}
		\Bigg( \sum_{|h| < H} \abs{\Ebis_{p\leq X} \chi^{h}(p)} \Bigg)^{d/2} \leq \big(e^{O(d)} + 2H \min(d^{-2}, e^{-j})\big)^{d/2} \leq e^{O(d^2)} + e^{O(d)}H^{d/2} \min(d^{-d}, e^{-jd/2}).
	\end{equation*}
	Summing this over all $\chi \in E_j$ and $j\geq 0$, we can bound \cref{eq:intermediatetoy} by
	\begin{equation*}
		\Ebis\!\big[\abs{\z}^2\big]^{1/2} \leq e^{O(d^2)}\phi(N)H^{d/2} + e^{O(d)}H^d \sum_{j\geq 0} \abs{E_j} \min(d^{-d}, e^{-jd/2}).
	\end{equation*}
	Note that $\varphi(N)\leq N\leq e^{d^2}$. Bounding $\abs{E_j}$ using \cref{eq:sizeEjtoy} and summing the resulting geometric series (using that $N$ is sufficiently large), we obtain
	\begin{equation*}
		\Ebis\!\big[\abs{\z}^2\big]^{1/2} \leq \big(e^{O(d^2)}H^{-d/2} + e^{O(d)}H d^{-d}\big) H^d.
	\end{equation*}
	Recalling that $H = \lfloor e^{Cd} \rfloor$ we see that, if $C$ is a sufficiently large absolute constant, the right-hand side is $o(H^d)$ as $N\to \infty$. This establishes \cref{eq:secondmomenttoy} and concludes the proof of \cref{prop:toyversion}.
\end{proof}

\begin{remark}
	\label{rem:aftertoyproof}
	For \cref{thm:shortbasis}, we obtain (with high probability) an asymptotic formula for quantities similar to \cref{eq:orthogonalitytoy}, which is then combined with techniques from geometry of numbers to prove the existence of a short lattice basis.

	This requires a more careful analysis of the small-order characters, which may be more numerous than in the special case treated above. Moreover, the analysis of large-order characters is much more delicate, as an exceptional character $\psi$ (in the sense of \cref{lem:ANTtoy}) may have more representations of the form $\psi = \chi^h$ (for some $\chi\in \widehat{G}$ and $|h|<H$) than the previous argument can handle. This issue is resolved by performing Fourier analysis on a well-chosen subgroup of $\mgp{N}$, introduced in \cref{sec:restrictionM}.
\end{remark}

\section{Proof of the existence of a short lattice basis}
\label{section:proof}

\subsection{Restricting to a convenient subgroup of \texorpdfstring{$\widehat{\mgp{N}}$}{the multiplicative group}}
\label{sec:restrictionM}

The goal of this preliminary section is to define a subgroup of the character group of $\mgp{N}$ that does not have too many elements of small order. This will be crucial for \cref{sec:largeorder}, to reduce the influence of potential counterexamples to the Generalised Riemann Hypothesis. For example, we need to avoid having many characters $\chi_1, \ldots, \chi_k$ modulo $N$ such that $\chi_i^2 = \psi$ for all $i$, where $\psi$ is an exceptional character (in the sense that $L(s, \psi)$ has a zero very close to the line $\Re s = 1$). The definition of the suitable subgroup depends on the precise structure of $\mgp{N}$.

\begin{notation}
	\label{not:group}
	If $G$ is a finite abelian group (written multiplicatively) and $M\geq 1$, we write $G^M := \{x^M : x\in G\}$ for the subgroup of $M$th powers in $G$.\footnote{In particular, $G^M$ is \emph{not} the $M$-fold direct product of $G$.}
\end{notation}

\begin{lemma}
	\label{lem:pontryagin}
	Let $G$ be a finite abelian group (written multiplicatively) and $M\geq 1$. There is an isomorphism
	\begin{equation*}
		\iota : \widehat{G^M} \xrightarrow{\sim} \widehat{G}^M
	\end{equation*}
	such that, for any $\chi\in \widehat{G^M}$, $\iota(\chi) = \widetilde{\chi}^M$ where $\widetilde{\chi}\in \widehat{G}$ is an arbitrary extension of $\chi$ to $G$.
\end{lemma}

\begin{proof}
	Pontryagin duality for finite abelian groups is an equivalence of categories, and therefore an exact functor. In simple terms, this implies that injections turn into surjections when switching to the dual, and vice versa. Hence, the $M$th power map $G \twoheadrightarrow G^M$, $g\mapsto g^M$, induces an injection $\iota : \widehat{G^M} \xhookrightarrow{} \widehat{G}$ defined by
	\begin{equation*}
		\iota(\chi)(g) := \chi\left(g^M\right)
	\end{equation*}
	for every $\chi \in \widehat{G^M}$ and $g\in G$. Moreover, the restriction map $\widehat{G} \to \widehat{G^M}$ is surjective, using exactness of Pontryagin duality again. Thus, every $\chi \in \widehat{G^M}$ can be extended to a character $\widetilde{\chi}$ on $G$, and
	\begin{equation*}
		\chi\left(g^M\right) = \widetilde{\chi}(g^M) =  \widetilde{\chi}^M(g).
	\end{equation*}
	This implies that the image of $\iota$ is contained in $\widehat{G}^M$. Since $\big\lvert\widehat{G^M}\big\rvert = \big\lvert{G^M}\big\rvert = \big\lvert\widehat{G}^M\big\rvert$, the lemma follows.
\end{proof}

\begin{definition}
	Let $N \geq 1$. For every $h\geq 1$ we define
	\begin{equation*}
		K(h) := \frac{\abs{\mgp{N}}}{\abs{(\mgp{N})^{h}}}.
	\end{equation*}
	In other words, $K(h)$ is the size of the kernel of the $h$th power map in $\mgp{N}$.
\end{definition}

\begin{lemma}
	\label{lem:defM}
	Let $N>2$ be an integer and $d := \lceil\! \sqrt{\log N} \rceil$.

	For every prime $p$, define $m_p$ to be the largest non-negative integer such that
	\begin{equation*}
		K(p^{m_p}) \geq p^{d m_p/10}.
	\end{equation*}
	Let $\theM = \theM(N) := \prod_{p} p^{m_p}$. Then the following holds.
	\begin{enumerate}
		\item We have $\theM \leq e^{10d}$.
		\item For every $h\geq 1$, we have $K(\theM h)/K(\theM)\leq h^{d/10}$.
	\end{enumerate}
\end{lemma}

\begin{proof}
	\begin{enumerate}
		\item For every prime $p$, since $K(p^{m_p})$ is a power of $p$ dividing $\abs{\mgp{N}} = \phi(N)$, we have
		      \begin{equation*}
			      \prod_{p} K(p^{m_p}) \mid \phi(N).
		      \end{equation*}
		      By definition of $m_p$, we know that $K(p^{m_p}) \geq p^{d m_p/10}$ for every $p$, and thus
		      \begin{equation*}
			      \theM^{d/10} = \prod_p p^{dm_p/10} \leq \prod_{p} K(p^{m_p}) \leq \phi(N) \leq N.
		      \end{equation*}
		      Recalling that $d = \lceil\! \sqrt{\log N} \rceil$, it follows that $\theM \leq e^{10d}$.
		\item In any finite abelian group (written multiplicatively) and for any coprime integers $a$ and $b$, there is an isomorphism
		      \begin{equation*}
			      \ker\!\big(g\mapsto g^{a}\big) \times \ker\!\big(g\mapsto g^{b}\big) \cong \ker\!\big(g\mapsto g^{ab}\big).
		      \end{equation*}
		      Thus, if $h = \prod_{p} p^{e_p}$ is the prime factorisation of $h$, we see that $K(\theM) = \prod_p K(p^{m_p})$ and $K(\theM h) = \prod_{p} K(p^{m_p+e_p})$, where all products are finite. Whenever $e_p\geq 1$, we have
		      \begin{equation*}
			      K(p^{m_p+e_p}) < p^{d(m_p+e_p)/10},\quad K(p^{m_p}) \geq p^{dm_p/10}
		      \end{equation*}
		      by definition of $m_p$. Thus ${K(p^{m_p+e_p})}/{K(p^{m_p})} \leq p^{de_p/10}$, and this last inequality also holds when $e_p = 0$. We conclude that
		      \begin{equation*}
			      \frac{K(\theM h)}{K(\theM)} = \prod_p \frac{K(p^{m_p+e_p})}{K(p^{m_p})} \leq \prod_p p^{de_p/10} = h^{d/10},
		      \end{equation*}
		      as claimed. \qedhere
	\end{enumerate}
\end{proof}

\begin{lemma}
	\label{lem:hroots}
	Let $N > 2$, $d := \lceil\! \sqrt{\log N} \rceil$, $G := \mgp{N}$ and let $\theM\geq 1$ be the integer defined in the statement of \cref{lem:defM}.

	Let $h\geq 1$. Let $\chi_1, \ldots, \chi_n \in \widehat{G^\theM}$ be distinct characters such that
	\begin{equation*}
		\chi_1^h = \chi_2^h = \ldots = \chi_n^h.
	\end{equation*}
	Then $n \leq h^{d/10}$.
\end{lemma}

\begin{proof}
	Let $f_h : G^\theM \to G^\theM$ be the $h$th power map, $f_h(\chi) := \chi^h$. By assumption, the $n$ distinct characters $\chi_1\chi_n^{-1}, \chi_2\chi_n^{-1}, \ldots, \chi_n\chi_n^{-1}$ lie in the kernel of $f_h$. Hence,
	\begin{equation*}
		n \leq \big\lvert \!\ker(f_h)\big\rvert = \frac{|G^\theM|}{|(G^\theM)^h|} = \frac{|(\mgp{N})^\theM|}{|(\mgp{N})^{\theM h}|} = \frac{K(\theM h)}{K(\theM)},
	\end{equation*}
	which is $\leq h^{d/10}$ by \cref{lem:defM}.
\end{proof}

\subsection{Lattice point counting via characters}
\label{sec:charactercount}

To prove \cref{thm:shortbasis}, we will need to study the following quantities.

\begin{definition}
	\label{def:Fchi}
	Let $N\geq 1$, $G = \mgp{N}$ and $\chi \in \widehat{G}$. For all $d\geq 1$ and $b_1, \ldots, b_d\in G$, we define the quantity
	\begin{equation*}
		F_{\chi,H}(b_1, \ldots, b_d) := \prod_{i=1}^d \sum_{|h|\leq H} \chi^h(b_i).
	\end{equation*}
\end{definition}

The following lemma relates these expressions to a lattice point counting problem.

\begin{lemma}
	\label{lem:countbychar}
	Let $N,M,d,H \geq 1$ be integers. Let $b_1, \ldots, b_d\in G := \mgp{N}$. The number of vectors $(e_1, \ldots, e_d)\in \Z^d \cap [-H,H]^d$ such that
	\begin{equation*}
		\prod_{i=1}^d b_i^{Me_i} \equiv 1 \pmod N
	\end{equation*}
	is equal to
	\begin{equation}
		\label{eq:countintochar}
		\frac{1}{\big\lvert \widehat{G}^M\big\rvert} \sum_{\chi \in \widehat{G}^M} F_{\chi,H}(b_1, \ldots, b_d).
	\end{equation}
\end{lemma}

\begin{proof}
	By orthogonality of characters in the abelian group $G^M$, we have
	\begin{equation*}
		\ind{\prod_i b_i^{Me_i} \equiv 1 \spmod{N}} = \frac{1}{\big\lvert \widehat{G^M}\big\rvert} \sum_{\psi \in \widehat{G^M}} \psi \left(\prod_{i=1}^d b_i^{Me_i}\right) = \frac{1}{\big\lvert \widehat{G^M}\big\rvert} \sum_{\psi \in \widehat{G^M}} \prod_{i=1}^d \psi \big( b_i^{Me_i}\big).
	\end{equation*}
	Observe that $\psi \big( b_i^{Me_i}\big) = \iota(\psi) (b_i^{e_i})$, where $\iota$ is the isomorphism $\widehat{G^M} \xrightarrow{\sim} \widehat{G}^M$ defined in \cref{lem:pontryagin}. We may thus rewrite this equality as
	\begin{equation*}
		\ind{\prod_i b_i^{Me_i} \equiv 1 \spmod{N}} = \frac{1}{\big\lvert \widehat{G}^M\big\rvert} \sum_{\chi \in \widehat{G}^M} \prod_{i=1}^d \chi \big( b_i^{e_i}\big).
	\end{equation*}
	Therefore, the number of vectors $(e_1, \ldots, e_d)\in \Z^d \cap [-H,H]^d$ such that $\prod_{i=1}^d b_i^{Me_i} \equiv 1 \pmod N$ is
	\begin{equation*}
		\sum_{\substack{(e_1, \ldots, e_d)\in \Z^d \\ \max_i |e_i|\leq H}} \ind{\prod_i \b_i^{Me_i} \equiv 1 \spmod{N}} = \frac{1}{\big\lvert \widehat{G}^M\big\rvert} \sum_{\chi \in \widehat{G}^M}  \sum_{\substack{(e_1, \ldots, e_d)\in \Z^d \\ \max_i |e_i|\leq H}}  \prod_{i=1}^d \chi^{e_i} (b_i) = \frac{1}{\big\lvert \widehat{G}^M\big\rvert} \sum_{\chi \in \widehat{G}^M} F_{\chi,H}(b_1, \ldots, b_d)
	\end{equation*}
	as claimed.
\end{proof}

It is fairly straightforward to estimate $F_{\chi,H}(b_1, \ldots, b_d)$ when $\chi$ is a character of small order in $\widehat{G}$. The contribution of these small-order characters gives the expected main term for \cref{eq:countintochar}.

\begin{lemma}[Main term]
	\label{lem:smallorderchar}
	Let $N,M,d \geq 1$ be integers. Let $b_1, \ldots, b_d\in G := \mgp{N}$. Uniformly for all integers $H\geq e^{31d}$, we have
	\begin{equation}
		\label{eq:smallorderchar}
		\frac{1}{\big\lvert \widehat{G}^M\big\rvert}\sum_{\chi \in \widehat{G}^M} F_{\chi,H}(\underline{b}) =  \left(1 + O\left(\frac{e^{31d}}{H}\right)\right) \frac{ (2H+1)^{d} }{ \big\lvert\inner{b_1^M, \ldots, b_d^M}\big\rvert } + O\Bigg(\frac{1}{\big\lvert \widehat{G}^M\big\rvert} \!\!\!\!\sum_{\substack{\chi \in \widehat{G}^M\\ \order(\chi) \geq e^{10d}}} \!\!\!\! \abs{F_{\chi,H}(\underline{b})}\Bigg),
	\end{equation}
	where $F_{\chi,H}(\underline{b})$ is short for $F_{\chi,H}(b_1, \ldots, b_d)$.
\end{lemma}

\begin{proof}
	Define
	\begin{equation*}
		\inner{b_1, \ldots, b_d}^{\perp} := \{\chi \in \widehat{G} \ :\  \forall i\in [d],\,\chi(b_i)=1 \}.
	\end{equation*}
	Observe that
	\begin{equation}
		\label{eq:Malgeq}
		\big\lvert \inner{b_1, \ldots, b_d}^{\perp} \cap \widehat{G}^M \big\rvert = \big\lvert \big\{\chi\in \widehat{G^M} \ :\ \forall i\in [d], \,\chi\big(b_i^M\big)=1 \big\}\big\rvert = \frac{\big\lvert {G^M}\big\rvert}{\big\lvert\inner{b_1^M, \ldots, b_d^M}\big\rvert}
	\end{equation}
	where the first equality follows from \cref{lem:pontryagin} and the second from the canonical isomorphism $H_1^{\perp} \cong \widehat{G_1/H_1}$ for any subgroup $H_1$ of a finite abelian group $G_1$.

	Clearly, if $\chi \in \inner{b_1, \ldots, b_d}^{\perp}\cap \widehat{G}^M$, we have $F_{\chi,H}(\underline{b}) = (2H+1)^d$. Hence, by \cref{eq:Malgeq}, those characters contribute
	\begin{equation*}
		\frac{1}{\big\lvert \widehat{G}^M\big\rvert} \sum_{\chi \in \inner{b_1, \ldots, b_d}^{\perp} \cap \widehat{G}^M} F_{\chi,H}(\underline{b}) = \frac{(2H+1)^d}{\big\lvert\inner{b_1^M, \ldots, b_d^M}\big\rvert}
	\end{equation*}
	to the sum \cref{eq:smallorderchar}.

	For any character $\chi \in \widehat{G}^{M}\setminus \inner{b_1, \ldots, b_d}^{\perp}$ of order $\leq e^{10d}$, and any $i\in [d]$ such that $\chi(b_i) \neq 1$, we can bound
	\begin{equation*}
		\abs{\sum_{|h| \leq H} \chi(b_i)^h} \leq \order(\chi) \leq e^{10d}
	\end{equation*}
	by the geometric series formula. Thus, defining
	\begin{equation*}
		I_{\chi} := \{i \in [d] : \chi(b_i) \neq 1\},
	\end{equation*}
	we have
	\begin{equation*}
		\abs{F_{\chi,H}(\underline{b})}  \leq e^{10d\abs{I_{\chi}}} (2H+1)^{d - \abs{I_{\chi}}}.
	\end{equation*}

	Consequently,
	\begin{equation}
		\label{eq:smallorderintermediate}
		\sum_{\substack{\chi \in \widehat{G}^{M}\setminus \inner{b_1, \ldots, b_d}^{\perp}\\ \order(\chi) \leq e^{10d}}} \abs{F_{\chi,H}(\underline{b})} \ll \sum_{\substack{I \subset [d]\\ I\neq \emptyset}} e^{10d\abs{I}} (2H+1)^{d - \abs{I}} \sum_{\substack{\chi \in \widehat{G}^M\\ \order(\chi) \leq e^{10d}\\ I_{\chi} = I}} 1.
	\end{equation}

	Fix some non-empty set $I \subset [d]$ and complex numbers $(z_i)_{i\in I}$. Let $C_{I, (z_i)}$ be the set of all characters $\chi \in \widehat{G}^M$ such that
	\begin{equation*}
		\chi(b_i) = \begin{cases}
			z_i & \text{if }i\in I                \\
			1   & \text{if }i\in  [d]\setminus I.
		\end{cases}
	\end{equation*}
	Note that $C_{I, (z_i)}$ is either the empty set, or a coset of $\inner{b_1, \ldots, b_d}^{\perp} \cap \widehat{G}^M$ in $\widehat{G}^M$. Moreover, if $\chi$ has order at most $e^{10d}$, this set $C_{I, (z_i)}$ can only be non-empty if all $z_i$ are roots of unity of order $\leq e^{10d}$, and there are $\leq e^{20d}$ such roots of unity. We conclude that
	\begin{equation*}
		\sum_{\substack{\chi \in \widehat{G}^M\\ \order(\chi) \leq e^{10d}\\ I_{\chi} = I}} 1 \leq \sum_{(z_i)_{i\in I}} \abs{C_{I, (z_i)}}\leq e^{20d|I|} \big\lvert \inner{b_1, \ldots, b_d}^{\perp} \cap \widehat{G}^M \big\rvert.
	\end{equation*}

	Thus, we can bound the right-hand side of \cref{eq:smallorderintermediate} by
	\begin{equation*}
		\ll (2H+1)^{d} \big\lvert \inner{b_1, \ldots, b_d}^{\perp} \cap \widehat{G}^M \big\rvert \sum_{\substack{I \subset [d]\\ I\neq \emptyset}} \big(e^{30d} H^{-1}\big)^{\abs{I}} \ll (2H+1)^{d} \big\lvert \inner{b_1, \ldots, b_d}^{\perp} \cap \widehat{G}^M \big\rvert  2^{d}e^{30d}H^{-1}
	\end{equation*}
	where we used that $H\geq e^{30d}$ in the last step. By \cref{eq:Malgeq}, this means that
	\begin{equation*}
		\frac{1}{\big\lvert \widehat{G}^M\big\rvert}\sum_{\substack{\chi \in \widehat{G}^{M}\setminus \inner{b_1, \ldots, b_d}^{\perp}\\ \order(\chi) \leq e^{10d}}} \abs{F_{\chi,H}(\underline{b})}\ll e^{31d}H^{-1} \frac{(2H+1)^d}{\big\lvert\inner{b_1^M, \ldots, b_d^M}\big\rvert},
	\end{equation*}
	which completes the proof.
\end{proof}

\subsection{Bounding the contribution of large-order characters}
\label{sec:largeorder}

In this section, we bound the error term coming from large-order characters, on average over primes $b_1, \ldots, b_d$ in a short interval.

We will need the following ingredients from classical analytic number theory.

\begin{proposition}[Character sums over primes]
	\label{lem:charbound}
	Let $1/2 \leq \alpha \leq 1$. Let $q, x\geq 2$. Let $\chi$ be a non-principal character modulo $q$ whose Dirichlet $L$-function $L(s, \chi)$ has no zero in the rectangle
	\begin{equation*}
		\rect{\alpha}{x^{1-\alpha}}.
	\end{equation*}
	Then
	\begin{equation*}
		\frac{1}{\pi(x)} \sum_{p\leq x} \chi(p) \ll x^{-(1-\alpha)} \log^3 (qx).
	\end{equation*}
\end{proposition}

\cref{lem:charbound} is a standard consequence of the explicit formula for $L(s, \chi)$ and is proved in \cref{sec:appendix}. The logarithmic factors can be somewhat improved, but such refinements are irrelevant here.

The Generalised Riemann Hypothesis is the claim that, for every Dirichlet character $\chi$, $L(s, \chi)$ has no zero $\rho$ with $\frac{1}{2} < \Re \rho < 1$. Assuming GRH, \cref{lem:charbound} thus implies almost square-root cancellation for character sums over primes. \Cref{prop:ZDEmodN} will serve as an unconditional substitute for GRH.

\begin{proposition}[Zero-density estimate for a fixed modulus]
	\label{prop:ZDEmodN}
	Uniformly for $\tfrac45 \leq \alpha \leq 1$ and $q, T\geq 1$, we have
	\begin{equation*}
		\sum_{\substack{\chi \spmod q}} \Nb(\alpha, T, \chi) \ll_{\eps} (qT)^{(2+\eps)(1-\alpha)},
	\end{equation*}
	where $\Nb(\alpha, T, \chi)$ denotes the number of zeros (with multiplicity) of $L(s, \chi)$ in the rectangle
	\begin{equation*}
		\rect{\alpha}{T}.
	\end{equation*}
\end{proposition}
\begin{proof}
	This is \cite[Theorem~1~(1.7)]{jutila}.
\end{proof}

Our goal is to control the total contribution of all large-order characters in \cref{eq:smallorderchar}. We will do so in \cref{lem:alllargeorderchar} (restricting to a suitable subgroup of $\widehat{G}$). We first prove the following $L^2$ bound for a single large-order character.

\begin{lemma}
	\label{lem:weakboundallchar}
	Let $N > 2$ be an integer and let $G = \mgp{N}$. Let $d = \lceil \sqrt{\log N} \rceil$ and $X =  d^{10^3 d}$.

	Let $\b_1, \ldots, \b_d$ be i.i.d.~random variables, each uniformly distributed in the set of primes less than~$X$ not dividing $N$. Let $\chi\in \widehat{G}$ be a character of order $\geq e^{10d}$. Then, for every $H\geq e^{10d}$,
	\begin{equation*}
		\E{ |F_{\chi,H}(\b_1, \ldots, \b_d)|^2 } \ll d^{-10d} H^{2d}.
	\end{equation*}
\end{lemma}

\begin{proof}
	By \cref{prop:ZDEmodN} (the zero-density estimate), the number of Dirichlet characters modulo~$N$ whose associated $L$\nobreakdash-function has a zero in the region
	\begin{equation}
		\label{eq:smallregion}
		\rect{1-\frac{1}{10d}}{X}
	\end{equation}
	is bounded by
	\begin{equation}
		\label{eq:fewexceptions}
		\ll (NX)^{(2+1)/(10d)}\ll e^{d}.
	\end{equation}

	If a non-principal character $\psi$ modulo $N$ has no zero in the region \cref{eq:smallregion}, then by \cref{lem:charbound} we have
	\begin{equation}
		\label{eq:weakcancellation}
		\E{\psi(\b_1)} \ll X^{-1/(10d)}(\log (NX))^3  \ll d^{-100} d^6 \ll d^{-94}.
	\end{equation}

	We now expand $\E{ |F_{\chi,H}(\b_1, \ldots, \b_d)|^2 }$ as
	\begin{align*}
		\E{ \bigg\lvert \prod_{i=1}^d \sum_{|h|\leq H} \chi^h(\b_i) \bigg\rvert^2 } & = \prod_{i=1}^d \E{\bigg\lvert\sum_{|h|\leq H} \chi^h(\b_i)\bigg\rvert^2 }     \\
		                                                                            & = \Bigg(\sum_{|h_1|\leq H} \sum_{|h_2|\leq H} \E{\chi^{h_1-h_2}(\b_1)}\Bigg)^d \\
		                                                                            & \leq (2H+1)^d \Bigg(\sum_{|h|\leq 2H} \abs{\E{\chi^h(\b_1)}}\Bigg)^d.
	\end{align*}
	We can use \cref{eq:weakcancellation} to bound the term $\E{\chi^h(\b_1)}$, unless $\chi^h$ is principal or $L(s, \chi^h)$ has a zero in the region \cref{eq:smallregion}. By \cref{eq:fewexceptions}, the number of values of $h$ with $|h|\leq 2H$ for which one of these two situations occurs is
	\begin{equation*}
		\ll \left(1+\frac{H}{\order(\chi)}\right)  e^{d} \ll e^{-10d} e^{d}H  \ll e^{-d}H.
	\end{equation*}
	By \cref{eq:weakcancellation}, we deduce that, for all but $O(e^{-d}H)$ values of $|h|\leq 2H$, the bound $\E{\chi^h(\b_1)}\ll d^{-94}$ holds. Hence
	\begin{equation*}
		\E{ |F_{\chi,H}(\b_1, \ldots, \b_d)|^2 } \leq (2H+1)^d \Big(O(e^{-d}H) + O(d^{-94}H)\Big)^d \ll e^{O(d)}d^{-94d} H^{2d} \ll d^{-10d} H^{2d},
	\end{equation*}
	as desired.
\end{proof}

\cref{lem:weakboundallchar} applies to all characters $\chi$ of order at least $e^{10d}$, but gives a relatively weak upper bound. In \cref{lem:largeorderchar}, we will prove that a stronger bound can be obtained if a small set of exceptions is allowed. We begin by describing the exceptional characters in the following lemma.

\begin{lemma}
	\label{lem:strongerboundlargechar}
	There is some absolute constant $t_0\geq 1$ such that the following holds.

	Let $N,G,d,X$ and $\b_i$ be as in \cref{lem:weakboundallchar}. Let $H\geq e^{10d}$.

	Let $\chi \in \widehat{G}$ be a character of order $\geq e^{10d}$ such that
	\begin{equation*}
		\E{ |F_{\chi,H}(\b_1, \ldots, \b_d)|^2 } > e^{-2td} H^{2d}
	\end{equation*}
	for some $t$ with $t_0 \leq t\leq 2d$.

	Then there are integers $-2H \leq h_1, h_2\leq 2H$ with $0 < h_2-h_1 \leq e^{3t}$ such that both $L(s, \chi^{h_1})$ and $L(s, \chi^{h_2})$ have a zero in the region
	\begin{equation}
		\label{eq:largerregion}
		\rect{1-\frac{t}{100d}}{X}.
	\end{equation}
\end{lemma}

\begin{proof}
	As in the proof of \cref{lem:weakboundallchar}, we have
	\begin{equation}
		\label{eq:ineqHd}
		\E{ |F_{\chi,H}(\b_1, \ldots, \b_d)|^2 } = \E{ \bigg\lvert \prod_{i=1}^d \sum_{|h|\leq H} \chi^h(\b_i) \bigg\rvert^2 } \leq (2H+1)^d \Bigg(\sum_{|h|\leq 2H} \abs{\E{\chi^h(\b_1)}}\Bigg)^d.
	\end{equation}
	Let $I_1$ be the set of all $-2H\leq h\leq 2H$ such that $\chi^h$ is principal. Since $\chi$ has order $\geq e^{10d}$, we have
	\begin{equation*}
		|I_1| \ll 1+e^{-10d}H \ll e^{-10d}H.
	\end{equation*}

	Let $I_2$ be the set of all $-2H\leq h\leq 2H$ such that $L(s, \chi^h)$ has a zero in the region \cref{eq:largerregion}. By contradiction, suppose that the conclusion of \cref{lem:strongerboundlargechar} does not hold. Then any sub-interval of $[-2H, 2H]$ of length $e^{3t}$ contains at most one element of $I_2$. This implies that
	\begin{equation*}
		\abs{I_2} \ll 1+e^{-3t}H \ll e^{-3t}H.
	\end{equation*}

	Moreover, for any integer $h\in [-2H, 2H]\setminus (I_1\cup I_2)$, the character $\chi^h$ is non-principal and has no zero in the region \cref{eq:largerregion}, which by \cref{lem:charbound} implies that
	\begin{equation*}
		\E{\chi^h(\b_1)} \ll X^{-t/(100d)}(\log (NX))^3 \ll d^{-10t} d^6  \ll e^{-3t}.
	\end{equation*}

	Therefore, we can bound
	\begin{equation*}
		\sum_{|h|\leq 2H} \abs{\E{\chi^h(\b_1)}} \ll |I_1| + |I_2| + e^{-3t}H \ll e^{-10d}H+e^{-3t}H \ll e^{-3t}H.
	\end{equation*}
	Hence, by \cref{eq:ineqHd} we get
	\begin{equation*}
		\E{ |F_{\chi,H}(\b_1, \ldots, \b_d)|^2 } \leq e^{O(d)}e^{-3td}H^{2d},
	\end{equation*}
	which contradicts the assumption in the statement if $t_0$ is chosen to be sufficiently large.
\end{proof}

We wish to use a zero-density estimate again to show that there are few characters satisfying the conclusion of \cref{lem:strongerboundlargechar}. For this step to work, we need to restrict to the subgroup $\widehat{G}^{\theM}$ of the full character group $\widehat{G}$, where $\theM$ is the integer defined in \cref{lem:defM}.

\begin{lemma}
	\label{lem:largeorderchar}
	Let $t_0\geq 1$ be the constant from \cref{lem:strongerboundlargechar}. Let $N,G,d,X$ and $\b_i$ be as in \cref{lem:weakboundallchar}. Let $H\geq e^{10d}$ and let $\theM\geq 1$ be the integer defined in \cref{lem:defM}.

	For every $t\geq t_0$,
	\begin{equation}
		\label{eq:setlevelt}
		\abs{ \big\{ \chi \in \widehat{G}^{\theM} \ :\ \order(\chi) \geq e^{10d},\ \E{ |F_{\chi,H}(\b_1, \ldots, \b_d)|^2 } > e^{-2td} H^{2d} \big\}} \ll e^{td/2}.
	\end{equation}
\end{lemma}

\begin{proof}
	If $t\geq 2d$, the bound \cref{eq:setlevelt} is trivially satisfied as the right-hand side is $\gg N$. We may thus assume that ${t_0 \leq t\leq 2d}$, in which case \cref{lem:strongerboundlargechar} applies.

	Let $E_t$ be the set of all characters $\psi$ modulo $N$ such that $L(s, \psi)$ has a zero in the rectangle defined in \cref{eq:largerregion}. By \cref{prop:ZDEmodN}, this set $E_t$ has size
	\begin{equation}
		\label{eq:setet}
		\abs{E_t} \ll (NX)^{(2+1/2)t/(100 d)}\ll e^{td/20}.
	\end{equation}

	For every $\chi \in \widehat{G}$ of order $\geq e^{10d}$, if $\E{ |F_{\chi,H}(\b_1, \ldots, \b_d)|^2 } > e^{-2td} H^{2d}$ then by \cref{lem:strongerboundlargechar} we can write
	\begin{equation*}
		\chi^h = \psi_1 \overline{\psi_2}
	\end{equation*}
	for some integer $0< h < e^{3t}$ and some characters $\psi_1, \psi_2\in E_t$. Hence, the left-hand side of \cref{eq:setlevelt} is
	\begin{equation}
		\label{eq:lo2}
		\leq \sum_{\psi_1,\psi_2\in E_t} \sum_{0 < h < e^{3t}} \abs{\big\{ \chi\in \widehat{G}^\theM \ : \ \chi^h = \psi_1\overline{\psi_2}\big\}}.
	\end{equation}
	Since $\widehat{G}^{\theM}$ and $\widehat{G^\theM}$ are isomorphic, \cref{lem:hroots} implies that
	\begin{equation}
		\label{eq:lo3}
		\abs{\big\{ \chi\in \widehat{G}^\theM \ : \ \chi^h = \psi_1\overline{\psi_2}\big\}} \leq h^{d/10}\leq e^{3td/10}.
	\end{equation}

	Inserting \cref{eq:setet,eq:lo3} into \cref{eq:lo2}, we conclude that the left-hand side of \cref{eq:setlevelt} is
	\begin{equation*}
		\leq |E_t|^2 e^{3t} e^{3td/10} \leq e^{3t} e^{4td/10} \ll e^{td/2}
	\end{equation*}
	as claimed.
\end{proof}

Combining the previous lemmas, we obtain a suitable bound for the sum of second moments of $F_{\chi,H}(\b_1, \ldots, \b_d)$ over all large-order characters in $\widehat{G}^{\theM}$.

\begin{proposition}[Large-order characters]
	\label{lem:alllargeorderchar}
	Let $N,G,d,X$ and $\b_i$ be as in \cref{lem:weakboundallchar}. Let $\theM\geq 1$ be the integer defined in \cref{lem:defM}. For $H\geq e^{10d}$, we have
	\begin{equation*}
		\sum_{\substack{\chi\in \widehat{G}^{\theM}\\ \order(\chi) \geq e^{10d}}} \E{ |F_{\chi,H}(\b_1, \ldots, \b_d)|^2}^{1/2} \ll d^{-2d} H^d.
	\end{equation*}
\end{proposition}

\begin{proof}
	By \cref{lem:weakboundallchar}, we know that $\E{ |F_{\chi,H}(\underline{\b})|^2}^{1/2} \leq C d^{-5d} H^d$ for all characters $\chi\in \widehat{G}$ of order $\geq e^{10d}$, where $C>0$ is an absolute constant (as before, $\underline{\b}$ stands for $\b_1, \ldots, \b_d$). Therefore,
	\begin{equation*}
		\sum_{\substack{\chi\in \widehat{G}^{\theM}\\ \order(\chi) \geq e^{10d}}} \E{ |F_{\chi,H}(\underline{\b})|^2}^{1/2} \leq \sum_{m = m_0}^{+\infty}\sum_{\substack{\chi\in \widehat{G}^{\theM}\\ \order(\chi) \geq e^{10d}}} e^{-m+1}H^d \ind{e^{-m} < \E{ |F_{\chi,H}(\underline{\b})|^2}^{1/2} H^{-d} \leq e^{-m+1}}(\chi)
	\end{equation*}
	where $m_0 := \lfloor 5d\log d - \log C + 1\rfloor$. We may assume that $m_0 \geq \max(t_0d, 4d\log d)$, since otherwise $d\ll 1$ and \cref{lem:alllargeorderchar} is trivially satisfied (given that $|\widehat{G}^\theM| \leq N \ll 1$ when $d\ll 1$).

	Applying \cref{lem:largeorderchar} with $t := m/d$, we obtain that for every $m \geq t_0d$,
	\begin{equation*}
		\abs{ \left\{\chi\in \widehat{G}^\theM \ :\  \order(\chi) \geq e^{10d},\ \E{ |F_{\chi,H}(\underline{\b})|^2}^{1/2} > e^{-m}H^{d}\right\} } \ll e^{m/2}.
	\end{equation*}
	Therefore,
	\begin{equation*}
		\sum_{\substack{\chi\in \widehat{G}^\theM\\ \order(\chi) \geq e^{10d}}} \E{ |F_{\chi,H}(\underline{\b})|^2}^{1/2} \ll \sum_{m = m_0}^{+\infty} e^{m/2} e^{-m+1} H^d \ll e^{-m_0/2} H^d
	\end{equation*}
	which is $O(d^{-2d} H^d)$ as $m_0 \geq 4d\log d$.
\end{proof}

\subsection{Geometry of numbers}
\label{sec:geometrynumbers}

In this section, we show how to pass from good estimates on the number of lattice points in certain regions to the existence of a short basis for the lattice.

\begin{lemma}
	\label{lem:cubes}
	Let $L\geq 1$ be an integer. Cover the cube $[-L,L]^d$ by $(2L)^d$ cubes of side length $1$ in the obvious way. Label these unit cubes $C_1, \ldots, C_{(2L)^d}$ (in any order). Let $V\subset \R^d$ be a hyperplane through the origin. Then the number of unit cubes $C_i$ intersecting $V$ is at most $(d+1) (2L)^{d-1}$.
\end{lemma}

\begin{proof}
	Let $e_1, \ldots, e_d$ be the standard basis for $\R^d$. Let $v\in \R^d$ be a unit vector orthogonal to $V$. Without loss of generality, since $\norm{v}_2 = 1$, we may assume that $\inner{v, e_1} \geq 1/\sqrt{d}$.

	Suppose that $V$ intersects two unit cubes $C_i$ and $C_j$ where $C_i = C_j + ke_1$ for some integer $k \geq 0$. Then, there is some $p\in C_i$ and some $w\in \R^d$ with $\norm{w}_{\infty} \leq 1$ such that $p\in V$ and $p+ke_1+w\in V$. Therefore, $ke_1+w \in V$ and thus $\inner{ke_1+w,v} = 0$. Noting that
	\begin{equation*}
		\inner{ke_1+w, v} \geq k \inner{e_1, v} - \norm{w}_{2} \norm{v}_{2} \geq \frac{k}{\sqrt{d}} - \sqrt{d},
	\end{equation*}
	we deduce that $k\leq d$.

	We have thus proved that, for any cube $C_i$, there are at most $d+1$ cubes of the form $C_i+ke_1$ for some $k\in \Z$ which intersect $V$. The lemma follows.
\end{proof}

The next lemma allows us to convert information about the number of lattice points in cubes into the existence of short linearly independent lattice vectors.

\begin{lemma}
	\label{lem:shortLIvectors}
	Let $d \geq 1$. Let $\Lambda \subset \R^d$ be a full-rank lattice. Let $2\leq H_0 < H_1$ be real numbers such that $H_1/H_0$ is an integer. Suppose that, for $i\in \{0, 1\}$,
	\begin{equation}
		\label{eq:countpoints}
		\abs{\Lambda \cap [-H_i, H_i]^d} = \theta_i \frac{(2H_i+1)^d}{\vol(\R^d/\Lambda)}
	\end{equation}
	where $\theta_0, \theta_1>0$ satisfy $\frac{H_1}{H_0} > \frac{\theta_0}{\theta_1} d \big(\tfrac{5}{2}\big)^d$. Then $\Lambda\cap [-H_1, H_1]^d$ contains $d$ linearly independent vectors.
\end{lemma}

\begin{proof}
	By contradiction, suppose that there exists a linear hyperplane $V$ containing all the points $v\in \Lambda \cap [-H_1, H_1]^d$.

	Let $L = H_1/H_0$. The large cube $[-H_1, H_1]^d$ can be covered by $(2L)^d$ axis-parallel cubes of side length $H_0$ in the natural way. By \cref{lem:cubes}, we can bound
	\begin{equation*}
		\abs{\Lambda \cap [-H_1, H_1]^d} \leq (d+1) (2L)^{d-1} \sup_C~\abs{\Lambda \cap C},
	\end{equation*}
	where the supremum runs over all (not necessarily centred) axis-parallel cubes $C\subset \R^d$ of side length $H_0$. If $C$ is such a cube and $v\in \Lambda\cap C$, then $C \subset v + [-H_0, H_0]^d$, which implies that
	\begin{equation*}
		\abs{\Lambda \cap C} \leq \abs{\Lambda \cap \big(v + [-H_0, H_0]^d\big)} = \abs{\Lambda \cap [-H_0, H_0]^d}.
	\end{equation*}
	Thus,
	\begin{equation*}
		\abs{\Lambda \cap [-H_1, H_1]^d} \leq (d+1) (2L)^{d-1}\abs{\Lambda \cap [-H_0, H_0]^d}.
	\end{equation*}
	Plugging in our lattice point estimate \cref{eq:countpoints}, we get
	\begin{equation*}
		\theta_1 (2H_1+1)^{d} \leq (d+1) (2L)^{d-1} \theta_0 (2H_0+1)^{d}.
	\end{equation*}
	Using $2H_1+1 \geq 2LH_0$ and $d+1 \leq 2d$, this implies that $L \leq \frac{\theta_0}{\theta_1} d \big(2+\frac{1}{H_0}\big)^d$, contradicting the inequality in the statement of the lemma.
\end{proof}

The linearly independent vectors given by \cref{lem:shortLIvectors} can be upgraded to a genuine basis for $\Lambda$ by standard geometry of numbers, namely Mahler's theorem. We state this fact in a slightly more general situation.

\begin{lemma}
	\label{lem:Mahler}
	Let $d\geq 1$. Let $\Lambda_1, \Lambda_2 \subset \R^d$ be full-rank lattices such that $M \Lambda_1 \subset \Lambda_2$ for some integer $M\geq 1$. Suppose that $\Lambda_1$ contains $d$ linearly independent vectors in $[-H, H]^d$ for some $H>0$. Then, $\Lambda_2$ admits a basis where each basis vector has Euclidean norm $\leq d^{3/2}MH$.
\end{lemma}

\begin{proof}
	By assumption, there are linearly independent vectors $v_1, \ldots, v_d \in \Lambda_1 \cap [-H, H]^d$. Then $Mv_1, \ldots, Mv_d$ are linearly independent vectors of $\Lambda_2$ such that $\max_{i\in [d]} \norm{Mv_i}_2 \leq \sqrt{d} MH$.

	Let $B\subset \R^d$ be the unit ball for the Euclidean norm. Let~${0 < \lambda_1 \leq \cdots \leq \lambda_d}$ be the successive minima\footnote{See \cite[Definition~3.29]{taovu} for the definition of successive minima.} of $B$ with respect to $\Lambda_2$. Since $Mv_1, \ldots, Mv_d \in \Lambda_2$ are linearly independent, we deduce from the above that $\lambda_d \leq \sqrt{d} MH$.

	By Mahler's theorem (see \cite[Theorem~3.34]{taovu}), we conclude that $\Lambda_2$ admits a basis of vectors of Euclidean norm at most $d \lambda_d \leq d^{3/2}MH$, which is what we needed to show.
\end{proof}

\subsection{Short basis vectors}

We can now prove our main technical result, which may be of independent interest.

\begin{theorem}
	\label{thm:shortbasis}
	Let $N>2$ be an integer. Let $d := \lceil\!\sqrt{\log N} \rceil$ and $X := d^{10^3d}$.

	Let $\b_1, \ldots, \b_d$ be i.i.d.~random variables, each uniformly distributed in the set of primes $\leq X$ not dividing $N$. Let $r\geq 0$ and let $\x_1, \ldots, \x_r$ be arbitrary\footnote{In particular, the $\x_i$ are not assumed to be independent or identically distributed.} random variables taking values in $\mgp{N}$.

	Then, with probability $1 +O\big(d^{-d}\big)$, the lattice
	\begin{equation*}
		\L := \Big\{ (e_1, \ldots, e_d, f_1, \ldots, f_r) \in \Z^{d+r} \ : \ \prod_{i=1}^d \b_i^{e_i} \prod_{i=1}^r \x_i^{f_i} \equiv 1 \!\!\pmod{N} \Big\}
	\end{equation*}
	has a basis consisting of vectors of Euclidean norm $\ll e^{42(d+r)}$.
\end{theorem}

\begin{remark}
	\label{rem:quarantedeux}
	The constant $42$ in the exponent is by no means the limit of our techniques. For example, using smooth cutoffs in \cref{def:Fchi} would significantly reduce the error term in \cref{lem:smallorderchar} and hence lower this constant. We have not performed such optimisations to keep the paper as simple as possible.
\end{remark}

\begin{proof}[Proof of \cref{thm:shortbasis}]
	Let $M := \theM(N)$ be the integer defined in \cref{lem:defM}. We introduce the auxiliary lattice
	\begin{equation*}
		\L_{M} :=  \Big\{ (e_1, \ldots, e_d, f_1, \ldots, f_r) \in \Z^{d+r} \ : \ \prod_{i=1}^d \b_i^{M e_i} \prod_{i=1}^r \x_i^{M f_i} \equiv 1 \!\!\pmod{N} \Big\}.
	\end{equation*}
	By \cref{lem:countbychar,lem:smallorderchar}, we have the lattice point estimate
	\begin{equation}
		\label{eq:latptest}
		\abs{\L_{M} \cap [-H, H]^{d+r}} = \frac{  \left(1 + O\left({e^{31(d+r)}}{H^{-1}}\right)\right) (2H+1)^{d+r} }{ \big\lvert\inner{ \b_1^{M}, \ldots, \b_d^{M}, \x_1^{M}, \ldots, \x_r^M }\big\rvert } + O\Bigg( \frac{1}{\big\lvert \widehat{G}^{M}\big\rvert} \!\!\!\!\sum_{\substack{\chi \in \widehat{G}^M\\ \order(\chi) \geq e^{10d}}} \!\!\!\! \abs{F_{\chi,H}(\underline{\b}, \underline{\x})} \Bigg)
	\end{equation}
	for every integer $H\geq e^{31(d+r)}$, where
	\begin{equation*}
		F_{\chi,H}(\underline{\b}, \underline{\x}) := F_{\chi, H}(\b_1, \ldots, \b_d, \x_1, \ldots, \x_r) = \Bigg(\prod_{i=1}^d \sum_{|h|\leq H} \chi^h(\b_i)\Bigg) \Bigg(\prod_{i=1}^r \sum_{|h|\leq H} \chi^h(\x_i)\Bigg).
	\end{equation*}
	Since the $\x_i$ have unknown distributions, we will use the trivial bound
	\begin{equation*}
		\abs{F_{\chi,H}(\underline{\b},\underline{\x})} \leq (2H+1)^r \abs{F_{\chi,H}(\b_1, \ldots, \b_d)}.
	\end{equation*}
	Applying \cref{lem:alllargeorderchar}, we deduce that, for all $H\geq e^{10d}$,
	\begin{equation*}
		\sum_{\substack{\chi\in \widehat{G}^{\theM}\\ \order(\chi) \geq e^{10d}}} \E{ |F_{\chi,H}(\underline{\b}, \underline{\x})|^2}^{1/2} \ll d^{-2d} (2H+1)^{d+r}.
	\end{equation*}
	By the $L^2$ triangle inequality and Chebyshev's inequality, this implies that, for fixed $H\geq e^{10d}$,
	\begin{equation}
		\label{eq:proberterm}
		\Prbis\bigg({\sum_{\substack{\chi \in \widehat{G}^M\\ \order(\chi) \geq e^{10d}}} \!\!\!\! \abs{F_{\chi,H}(\underline{\b}, \underline{\x})}} > d^{-d}(2H+1)^{d+r}\bigg) \ll d^{-2d}.
	\end{equation}

	Let $H_0 := d \lceil e^{31(d+r)}\rceil$ and $H_1 := \lceil d (d+r) (5/2)^{d+r}\rceil H_0$. We apply \cref{eq:latptest,eq:proberterm} twice, once with $H = H_0$ and once with $H = H_1$ to obtain the following: with probability ${1+O\big(d^{-d}\big)}$, the two estimates
	\begin{equation*}
		\abs{\L_{M} \cap [-H_0, H_0]^{d+r}} = \left(1 + O(1/d)\right) \frac{(2H_0+1)^{d+r}}{ \big\lvert\inner{ \b_1^{M}, \ldots, \b_d^{M}, \x_1^{M}, \ldots, \x_r^M }\big\rvert}
	\end{equation*}
	and
	\begin{equation*}
		\abs{\L_{M} \cap [-H_1, H_1]^{d+r}} = \left(1 + O(1/d)\right) \frac{(2H_1+1)^{d+r}}{ \big\lvert\inner{ \b_1^{M}, \ldots, \b_d^{M}, \x_1^{M}, \ldots, \x_r^M }\big\rvert}
	\end{equation*}
	simultaneously hold.

	We now apply \cref{lem:shortLIvectors}. By our choice of $H_0$ and $H_1$, the inequality in the statement is satisfied provided that $d$ is sufficiently large. We can assume that $d$ is large enough as, for $d\ll 1$, we have $N \ll 1$ and \cref{thm:shortbasis} is trivially true. Thus, \cref{lem:shortLIvectors} implies that ${\L_M \cap [-H_1, H_1]^{d+r}}$ contains $d+r$ linearly independent vectors, with probability ${1+O\big(d^{-d}\big)}$. By \cref{lem:Mahler}, since $M\L_M \subset \L$, we conclude that, with probability ${1+O\big(d^{-d}\big)}$, $\L$ admits a basis of vectors of Euclidean norm at most
	\begin{equation*}
		\ll d^{3/2} M H_1 \ll d^{3/2} e^{10d} d^2 (d+r) (5/2)^{d+r} e^{31(d+r)} \ll e^{42(d+r)}
	\end{equation*}
	using the bound $M \leq e^{10d}$ given by \cref{lem:defM}. This completes the proof.
\end{proof}

\section{Applications to quantum computing}
\label{sec:proofsquantum}

In this section, we prove the correctness of efficient quantum algorithms for factoring and for the discrete logarithm problem by applying our version of Regev's number-theoretic conjecture, \cref{thm:shortbasis}.

\subsection{Preparatory lemmas}

\begin{lemma}
	\label{lem:classicalmult}
	Let $N, d, m \geq 2$ be integers. There is a classical algorithm that, given integers ${0\leq a_1, \ldots, a_d\leq 2^{m}}$ and exponents $t_1, \ldots, t_d \in \{0, 1\}$, computes the product
	\begin{equation*}
		\prod_{i=1}^d a_i^{t_i} \spmod{N}
	\end{equation*}
	in time $O\big(md (\log d) \log(md)\big)$.
\end{lemma}

\begin{proof}
	This is similar to \cite[p5]{regev} or \cite[Lemma~12]{MIT} but for  general values of $m$. Without loss of generality, assume that $d$ is a power of $2$, say $d = 2^l$. We proceed to compute this product in a binary tree fashion. Let $T(k)$ be the complexity of multiplying any $k$ of these numbers $a_i^{t_i}$ modulo~$N$. Note that $T(2k) \leq 2T(k) + O(M(mk))$ where $M(x)$ is the time needed to multiply two integers having at most $x$ bits. By the work of Harvey and van der Hoeven \cite{annalsmult}, it is known that $M(x) = O(x\log x)$, which leads to the bound
	\begin{equation*}
		T(d) \ll \sum_{j=0}^{l} 2^j M(md/2^{j+1}) \ll \sum_{j=0}^{l} md \log(md/2^{j+1}) \ll md (\log d) \log(md),
	\end{equation*}
	as claimed.
\end{proof}%previously \cite[Lemma~5.6]{MIT}

We can turn this into a quantum circuit with the following well-known fact.

\begin{lemma}
	\label{lem:quantumconversion}
	Any classical circuit can be ``compiled'' into a reversible quantum circuit that carries out the same computations, with the number of gates and qubits used being proportional to the size of the classical circuit.
\end{lemma}

\begin{proof}
	This well-known fact is explained in Section~A.1 of the full version of \cite{MIT}.%previously \cite[Section~A.1]{MIT}
\end{proof}

\begin{lemma}
	\label{lem:generateprimes}
	Let $N$ be a sufficiently large integer. Let $d := \lceil \! \sqrt{\log N}\rceil$ and $X = d^{10^3d}$. Let $k = d^4$.

	Let $\n_1, \ldots, \n_k$ be i.i.d.~random variables uniformly distributed in $\{1, \ldots, X\}$. Then, the probability that at least $d$ of these $\n_i$ are prime numbers not dividing $N$ is $1 +O(1/N)$.
\end{lemma}

\begin{proof}
	Since $N$ has $\ll \log N$ prime factors, and since the number of primes $\leq X$ is $\gg X/\log X$, we have
	\begin{equation*}
		\Pr{\n_1 \text{ is prime and }\n_1\nmid N} \geq \frac{c}{\log X}
	\end{equation*}
	for some absolute constant $c>0$. Thus, if $E_i$ is the event that $\n_i$ is not prime or divides $N$, then by the union bound and independence,
	\begin{equation*}
		\Prbis \bigg(\bigcup_{\substack{I \subset [k]\\ |I|> k-d}} \bigcap_{i\in I} E_i\bigg) \leq \sum_{l< d} \binom{k}{l} \left(1-\frac{c}{\log X}\right)^{k-l} \leq dk^d e^{-c(k-d)/\log X} \ll e^{-d^2}
	\end{equation*}
	as needed.
\end{proof}

\subsection{The discrete logarithm problem}

We now prove \cref{thm:logarithm}. For convenience, we restate it here.

\thmlogarithm*

\begin{proof}[Proof of \cref{thm:logarithm}]
	Suppose we are given an integer $N>2$ and elements $g,y\in \mgp{N}$ such that $y$ lies in the subgroup generated by $g$ in $ \mgp{N}$. Let $d := \lceil\!\sqrt{\log N} \rceil$ and $X := d^{10^3d}$. Let $\b_1, \ldots, \b_d$ be i.i.d.~random variables, each uniformly distributed in the set of primes $\leq X$ not dividing $N$.

	Consider the random lattice
	\begin{equation*}
		\L_{\underline{\b},g,y} := \Big\{(e_1, \ldots, e_d, f_1, f_2) \in \Z^{d+2} \ : \ \Bigg(\prod_{i=1}^d \b_i^{e_i}\Bigg) g^{f_1} y^{f_2} \equiv 1 \!\!\pmod{N}\Big\}.
	\end{equation*}
	By \cref{thm:shortbasis}, with probability $1 +O(d^{-d})$, this lattice has a basis consisting of vectors of Euclidean norm $\ll e^{42d}$.

	Computing the discrete logarithm of $y$ with respect to the base $g$ reduces (with a polynomial-time classical algorithm) to computing a short basis for $\L_{\underline{\b},g,y}$. To see why this is the case, suppose that we managed to compute a basis $v_1, \ldots, v_{d+2}$ for $\L_{\underline{\b},g,y}$ with $\max_i \norm{v_i}_2 \ll e^{42d}$. It is then easy to find a vector in $\L_{\underline{\b},g,y}$ of the form $(0, \ldots, 0, x, -1)$ for some integer $x$ (note that such a vector exists since we assume that $y\in \inner{g}$). Indeed, this amounts to solving a linear system with integer coefficients. The complexity of solving an integer linear system is polynomial in the dimensions of the matrix and the number of bits of the coefficients, which are both $O(d)$ since $\max_i \norm{v_i}_2 \ll e^{42d}$ (see \cite{bareiss}). This yields an integer $x$ such that $g^x \equiv y \pmod{N}$.

	The algorithm for solving the discrete logarithm problem thus proceeds as follows.
	\begin{enumerate}
		\item Generate $d$ primes $b_1, \ldots, b_d$ independently and uniformly at random in the set of primes $\leq X$ not dividing $N$. To do this, it suffices to generate $d^4$ independent random integers $\leq X$. By \cref{lem:generateprimes}, the required conditions will be satisfied for at least $d$ of those with probability $1+O(1/N)$. This first step takes polynomial time on a classical computer using the AKS primality test~\cite{aks}.
		\item Use the procedure described by Eker{\aa} and G{\"a}rtner~\cite{discretelogs} to obtain a basis for the lattice $\L_{\underline{b},g,y}$. This involves making $O(d)$ calls to a quantum circuit, followed by polynomial-time classical post-processing. This step is now guaranteed to succeed with probability $\Theta(1)$, using the fact that $\L_{\underline{b},g,y}$ has a basis of vectors of Euclidean norm $\ll e^{42d}$ with probability $1 +O(d^{-d})$ (\cref{thm:shortbasis}).
		\item Compute the discrete logarithm of $y$ in classical polynomial time using this basis, as explained above.
	\end{enumerate}

	It remains to analyse the gate and space costs of the quantum part of this algorithm.

	The only modification needed to the analysis of the quantum circuit in~\cite{discretelogs} comes from the fact that $b_1, \ldots, b_d$ are not quite as small as in \cite{discretelogs}. In \cite{discretelogs} (and more generally in \cite{MIT,regev}), the primes $b_i$ are assumed to have $O(\log d)$ bits. In the present situation, we instead have the bound $b_i\leq X = d^{10^3d}$, i.e.~each $b_i$ has $O(d\log d)$ bits.

	The only place in \cite{discretelogs} where the assumption on the size of the $b_i$'s is used is in \cite[Lemma~3]{discretelogs}. In turn, the only point in the proof of \cite[Lemma~3]{discretelogs} where this assumption is needed is when \cite[Lemma~5]{MIT} is invoked (as a black box). The proof of \cite[Lemma~5]{MIT} is given in \cite[Section~5]{MIT}, and the size assumption on the $b_i$'s only comes up in \cite[Lemma~12]{MIT}.%previously \cite[Lemma~4.1]{MIT} twice, Section 5 and Lemma 5.6 twice (incl below)

	The result \cite[Lemma~12]{MIT} essentially\footnote{There are extra details pertaining to the precise use of qubits (e.g.~restoring the ancilla qubits to $|0\rangle$), but these will be the same in our setup.} states that there is a quantum circuit using $O(d \log^3d)$ gates and $O(d \log^3d)$ qubits to perform the computation of $\prod_{i=1}^d a_i^{t_i}$ where $t_i\in \{0, 1\}$ and $a_i$ are integers with $O(\log d)$ bits. This is a special case of \cref{lem:classicalmult} (used together with \cref{lem:quantumconversion}) applied with $m = O(\log d)$. In our case, we just apply \cref{lem:classicalmult} with $m = O(d\log d)$, together with  \cref{lem:quantumconversion} to convert the classical circuit into a quantum one. This yields a quantum circuit having $O(d^2\log^3 d)$ gates and $O(d^2\log^3 d)$ qubits to compute $\prod_{i=1}^d b_i^{t_i}$.

	The number of gates of the quantum circuit \cite[Lemma~5]{MIT} is%previously lemma 4.1, then below lemma 5.6 and algorithm 5.2 (stays)
	\begin{equation*}
		O\big( d (n\log n + d\log^3 d) \big) = O(n^{3/2} \log n)
	\end{equation*}
	(because \cite[Lemma~12]{MIT} is used $O(d)$ times, see \cite[Algorithm~5.2]{MIT} and the surrounding explanations). Note that this is the same as for Regev's algorithm (see \cite[p5]{regev}). In our case, the number of gates is
	\begin{equation*}
		O\big( d (n\log n + d^2\log^3 d) \big) = O(n^{3/2} \log^3 n).
	\end{equation*}

	The space optimisations of Ragavan and Vaikuntanathan keep the total number of qubits for their circuit \cite[Lemma~5]{MIT} under%previously Lemma~4.1
	\begin{equation*}
		O\big(n\log n + d\log^3 d\big) = O(n\log n),
	\end{equation*}
	due to the way the qubits are used and restored in the main loop of \cite[Algorithm~5.2]{MIT}. In our situation, the number of qubits is
	\begin{equation*}
		O\big(n\log n + d^2\log^3 d\big) = O(n\log^3 n).
	\end{equation*}

	In summary, our final quantum circuit has $O(n^{3/2}\log^3 n)$ gates and $O(n \log^3 n)$ qubits. As noted in~\cite{discretelogs}, the fact that the two elements $g$ and $y$ are not small ($g$ and $y$ can be as large as $N$, as opposed to the $b_i$'s) does not affect these complexity bounds.
\end{proof}

\begin{remark}
	\label{rem:spaceopti}
	In the proof of \cref{thm:logarithm}, we used the naive \cref{lem:quantumconversion} to convert a classical circuit into a quantum one. As mentioned in \cref{rem:saveonemorelog}, it is possible to save a factor of $\log n$ in the space cost for \cref{thm:logarithm,thm:factoring} by reusing certain qubits when performing the quantum computation corresponding to \cref{lem:classicalmult}.
\end{remark}

\subsection{Factoring integers}

In this section, we prove \cref{thm:factoring}.

\begin{lemma}
	\label{lem:shor}
	Let $N>1$ be an odd integer with at least two distinct prime factors. Let $\x$ be a random variable, uniformly distributed in $\mgp{N}$. Then $\inner{\x}$ contains a non-trivial square root of $1$ modulo $N$ with probability $\geq 1/2$.
\end{lemma}
\begin{proof}
	This elementary number-theoretic fact is well-known -- it was already needed for Shor's algorithm \cite{shor}. See \cite[Appendix~A4.3]{CNbook} for a detailed proof.
\end{proof}

\thmfactoring*

\begin{proof}[Proof of \cref{thm:factoring}]
	We can assume that $N$ is odd and not a perfect prime power, as otherwise it is easy to factor $N$ in polynomial time with a classical computer (see \cite[Exercise~5.17]{CNbook}).

	Let $d := \lceil\!\sqrt{\log N} \rceil$ and $X := d^{10^3d}$. The probabilistic algorithm to find a non-trivial divisor of $N$ goes as follows.

	\begin{enumerate}
		\item Generate $d$ primes $b_1, \ldots, b_d$ independently and uniformly at random in the set of primes $\leq X$ not dividing $N$. Sample another integer $x$ uniformly chosen in $\mgp{N}$. As in the proof of \cref{thm:logarithm}, this can be done classically in polynomial time with probability $1+O(1/N)$.
		\item Use the algorithm of Eker{\aa} and G{\"a}rtner~\cite{discretelogs} to obtain a basis for the lattice
		      \begin{equation*}
			      \L_{\underline{b},x} := \Big\{(e_1, \ldots, e_d, f) \in \Z^{d+1} \ : \ \Bigg(\prod_{i=1}^d b_i^{e_i}\Bigg) x^f \equiv 1 \!\!\pmod{N}\Big\}.
		      \end{equation*}
		      This involves making $O(d)$ calls to a quantum circuit, followed by polynomial-time classical post-processing. By \cref{thm:shortbasis}, $\L_{\underline{b},x}$ has a basis of vectors of Euclidean norm $\ll e^{42d}$ with probability $1 +O(d^{-d})$, which means that this step is guaranteed to succeed with probability~$\Theta(1)$.
		\item Using the short basis for $\L_{\underline{b},x}$ computed in the previous step, find the vector of the form $(0, \ldots, 0, r)$ in $\L_{\underline{b},x}$ with $r\geq 1$ as small as possible. This involves solving a linear system with integer coefficients, which can be done efficiently as in the proof of \cref{thm:logarithm}. This integer $r$ is the order of $x$ in $\mgp{N}$. By \cref{lem:shor}, with probability $\geq 1/2$, this order $r$ will be even and the element $x^{r/2}$ will be a non-trivial square root of $1$ modulo~$N$. Hence, $N$ divides the product $(x^{r/2}-1)(x^{r/2}+1)$ but neither term individually, which implies that $\gcd(N,\, x^{r/2}-1 \!\!\mod{N})$ is a non-trivial divisor of $N$.
	\end{enumerate}

	The analysis of this algorithm is identical to the corresponding part of the proof of \cref{thm:logarithm}.
\end{proof}

\appendix
\begin{appendices}
	\section{Bounds for character sums over primes}
	\label{sec:appendix}

	In this appendix, we prove \cref{lem:charbound}. We start by stating the truncated explicit formula for $L(s, \chi)$.

	\begin{lemma}
		\label{lem:explicitformula}
		Let $q\geq 2$. Let $\chi$ be a non-principal Dirichlet character modulo $q$. For $x\geq T\geq 2$, we have
		\begin{equation*}
			\sum_{n\leq x} \chi(n) \Lambda(n) = -\sum_{\substack{\rho = \beta+i\gamma\\ \abs{\gamma} \leq T}} \frac{x^{\rho}-1}{\rho} + O\left(\frac{x \log^2(xq)}{T}\right)
		\end{equation*}
		where the sum runs over all non-trivial zeros $\rho$ of $L(s, \chi)$ with multiplicity.
	\end{lemma}
	\begin{proof}
		This is \cite[Theorem~11.3]{kouk}.
	\end{proof}

	We will also need the following standard bound for the number of zeros of $L(s, \chi)$ in the critical strip at some height $t$.
	\begin{lemma}
		\label{lem:nzeros}
		Let $q\geq 1$ and let $\chi$ be a Dirichlet character modulo $q$. Let $t\in \R$. The number of zeros $\rho = \beta + i \gamma$ of $L(s, \chi)$ in the rectangle $0\leq \beta\leq 1$, $t\leq \gamma \leq t+1$ is $\ll \log(q(|t|+2))$, where zeros are counted with multiplicity.
	\end{lemma}
	\begin{proof}
		This is \cite[Theorem~10.17]{MV}.
	\end{proof}

	\begin{proof}[Proof of \cref{lem:charbound}]
		By \cref{lem:explicitformula,lem:nzeros}, choosing $T = x^{1-\alpha} - 1$, we have
		\begin{equation*}
			\sum_{n\leq x} \chi(n) \Lambda(n) \ll \log(qx)  \sum_{\substack{t\in \Z\\ |t| \leq x^{1-\alpha}-1}} \max_{\substack{\rho = \beta+i\gamma\\ \abs{\gamma - t} \leq 1/2}} \abs{\frac{x^{\rho}-1}{\rho}} + x^{\alpha}\log^2(qx),
		\end{equation*}
		where the maximum is over all zeros of $L(s, \chi)$ in the specified region. Since $\abs{x^\rho}\leq x^{\alpha}$ for all zeros $\rho = \beta + i\gamma$ with imaginary part $|\gamma|\leq x^{1-\alpha}$ by our zero-free rectangle assumption, we can bound
		\begin{equation*}
			\abs{\frac{x^{\rho}-1}{\rho}} \ll \frac{x^{\alpha}}{1+|\gamma|}
		\end{equation*}
		(using for example that $\frac{x^{\rho}-1}{\rho} = \int_1^x t^{\rho-1} dt$ for $|\gamma|<1$). Thus, we obtain
		\begin{equation*}
			\sum_{n\leq x} \chi(n) \Lambda(n) \ll  \log(qx) x^{\alpha} \sum_{|t|\leq x} \frac{1}{1+|t|}+ x^{\alpha}\log^2(qx)\ll x^{\alpha} \log^2(qx).
		\end{equation*}
		Discarding perfect prime powers, which contribute $O(x^{1/2}\log x)$, and using partial summation, we get
		\begin{equation*}
			\sum_{p\leq x} \chi(p) \ll \frac{x^{\alpha} \log^2(qx)}{(1-\alpha)\log x}.
		\end{equation*}
		We may assume that $1-\alpha \geq \frac{1}{\log x}$, as otherwise \cref{lem:charbound} is trivial. If this is the case, we conclude that
		\begin{equation*}
			\frac{1}{\pi(x)} \sum_{p\leq x} \chi(p) \ll \frac{x^{\alpha} \log^2(qx)}{\pi(x)} \ll x^{-(1-\alpha)} \log^3 (qx)
		\end{equation*}
		as claimed.
	\end{proof}
\end{appendices}

\bibliography{version_4}
\bibliographystyle{amsplain}

\end{document}